\documentclass[12pt]{article}
\usepackage{amsmath,amsfonts,amssymb,amsthm}
\newtheorem{theo}{Theorem}[section]
\newtheorem{prop}[theo]{Proposition}
\newtheorem{conj}[theo]{Conjecture}
\numberwithin{equation}{section}
\newcommand\la{\lambda}
\newcommand\al{\alpha}

\title{Jack polynomials and free cumulants}
\author{Michel Lassalle\\
\small Centre National de la Recherche Scientifique\\[-0.8ex]
\small Institut Gaspard-Monge, Universit\'e de Marne-la-Vall\'ee\\[-0.8ex]
\small 77454 Marne-la-Vall\'ee Cedex, France\\[-0.8ex]
\small \texttt{lassalle@univ-mlv.fr}\\[-0.8ex]
\small \texttt{http://igm.univ-mlv.fr/{\textasciitilde}lassalle}}
\date{}

\begin{document}
\maketitle
\begin{abstract}
We study the coefficients in the expansion of Jack polynomials in terms of power sums. We express them as polynomials in the free cumulants of the transition measure of an anisotropic Young diagram. We conjecture that such polynomials have nonnegative integer coefficients. This extends recent results about normalized characters of the symmetric group.

\end{abstract}

\section{Introduction}

The following connection between probability theory and representations of symmetric groups is due to Kerov~\cite{K1} and Biane~\cite{B1}. 

Let $S_n$ be the symmetric group of $n$ letters. Its irreducible representations are labelled by partitions of $n$, i.e. weakly decreasing sequences $\la= (\la_1,...,\la_l)$ of positive integers summing to $n$. We denote $l(\la)=l$ and $|\la|=n$, the respective length and weight of $\la$.

Given a partition $\la$, Kerov~\cite{K1} defined the ``transition measure'' of its Young diagram. The ``free cumulants'' of this measure are real numbers $\{R_i(\la), i\ge 2\}$ defined in the framework of free probability~\cite{V,Sp}. They were introduced by Biane~\cite{B1} to solve asymptotic problems in the representation theory of $S_n$.

Later Kerov~\cite{K2} and Biane~\cite{B2} observed that the character of the irreducible representation labelled by $\la$ may be written as a polynomial in the free cumulants $R_i(\la)$.

More precisely let $\textrm{dim}\,\la$ be the dimension of this representation and ${\chi}^\la_\rho=\chi^\la(\sigma)$ the value of its character at any permutation $\sigma\in S_n$ with cycle-type $\rho$.
Let $r\le n$ be a positive integer and  $\rho=(r,1^{n-r})$ the corresponding $r$-cycle in $S_n$. Kerov~\cite{K2} and Biane~\cite{B2} proved that the normalized character
\[\hat{\chi}^\la_r=n(n-1) \cdots (n-r+1) \frac{\chi^\la_{r,1^{n-r}}}{\textrm{dim}\,\la}\]
is a polynomial in the free cumulants $R_i(\la)$, with integer coefficients. They conjectured these coefficients to be nonnegative.

This property was recently proved by F\'eray~\cite{F}, who also extended it to
\[\hat{\chi}^\la_\mu=n(n-1) \cdots (n-r+1) \frac{\chi^\la_{\mu,1^{n-r}}}{\textrm{dim}\,\la},\]
with $\mu$ any partition of weight $r$ (see also~\cite{F2}). Previously, a method to compute $\hat{\chi}^\la_r$ had been given in~\cite{GR,R}, together with several explicit cases (see also~\cite{B3}). Stronger conjectures (yet unproved) were formulated in~\cite{La2}.

The purpose of this paper is to present an extension of these results in the framework of Jack polynomials. 

The family of Jack polynomials $J_{\la}(\alpha)$ is indexed by partitions. It forms a basis of the algebra of symmetric functions with rational coefficients in some positive real parameter $\alpha$. We consider the transition matrix between this basis and the classical basis of power sums $p_{\rho}$, i.e.
\[J_{\la}(\alpha)=\sum_{\rho} \theta^{\la}_{\rho}(\alpha) \,p_{\rho}.\]
Let $\mu$ be a partition with $r=|\mu|\le |\la|= n$. Using multiplicities, we write $\mu=(1^{m_1(\mu)},2^{m_2(\mu)},\ldots)$ and $z_\mu  = \prod_{i} i^{m_i(\mu)} m_i(\mu)!$. 

Firstly, we observe that the quantity 
\[\vartheta^{\la}_{\mu}(\alpha)=z_{\mu}\, \theta^{\la}_{\mu,1^{n-r}}(\alpha)\]
is a natural generalization of the normalized character $\hat{\chi}^\la_\mu$.
This is a consequence of the Frobenius formula for the Schur functions $s_\la$. Actually we have
\[s_{\la}=\sum_{\rho} z_\rho^{-1}\chi^\la_\rho \, p_\rho\quad \textrm{and}\quad J_{\la}(1)=h_\la s_\la,\]
where $h_\la=n!/\textrm{dim}\,\la$ is the hook-length of $\la$ (see~\cite[Examples 1.1.1 and 1.7.6]{Ma} and~\cite[p. 78]{S}). This yields
\begin{equation*}
z_{\rho}\, \theta^{\la}_{\rho}(1)=n!\,
\frac{\chi^\la_{\rho}}{\textrm{dim}\,\la},
\end{equation*}
and for $\rho=(\mu,1^{n-r})$,
\[ \hat{\chi}^\la_\mu=
\binom{n-r+m_1(\mu)}{m_1(\mu)}\,\vartheta^{\la}_{\mu}(1).\]
If $\mu$ is chosen with no part $1$, which is always possible, we thus have $\hat{\chi}^\la_\mu=\vartheta^{\la}_{\mu}(1)$. In particular $\hat{\chi}^\la_r=\vartheta^{\la}_{r}(1)$.

Secondly, expanding Jack polynomials in terms of the ``augmented'' monomial symmetric functions, which are integral combinations of the power sums~\cite[p. 110]{Ma}, and using~\cite[Theorem 1.1]{KS}, it is known that $\vartheta^{\la}_{\mu}(\alpha)$ is a polynomial in $\al$ with integer coefficients.

In this paper we consider the notion of ``anisotropic'' Young diagram of a partition $\la$, introduced by Kerov~\cite{K3}. This diagram is obtained from the classical Young diagram of $\la$ by the dilation $(u,v) \rightarrow (u,v/\al)$. 

We define the transition measure of this anisotropic Young diagram and the free cumulants $\{R_i(\la;\al), i\ge 1\}$ of this measure. We show that $-R_k(\la;\al)$ is a polynomial in $-1/\al$ with nonnegative integer coefficients. For instance, one has \[R_1(\la;\al)=0,\qquad
R_2(\la;\al)=|\la|/\al,\qquad
R_3(\la;\al)=\sum_i \la_i^2/\al-\sum_i (2i-1)\la_i/\al^2.\]

Given a partition $\mu$ having no part $1$, we present a method to express $\vartheta^{\la}_{\mu}(\alpha)$ in terms of the free cumulants $R_i(\la;\al)$. We show that there exists a polynomial $K_\mu$ such that for any partition $\la$ with $|\la|\ge |\mu|$, one has
\[\vartheta^{\la}_{\mu}(\al)=K_\mu\big(R_2(\la;\al),R_3(\la;\al),\ldots,R_{|\mu|-l(\mu)+2}(\la;\al)\big).\]

The coefficients of $K_\mu$ in the free cumulants are only known to be in $\mathbf{Q}(\alpha)$, the field of rational functions in $\al$. However applying our algorithm, extensive computer calculations support the following conjecture.
\begin{conj}
The coefficients of $K_\mu$ are polynomials in $\al$ with integer coefficients.
\end{conj}

Following F\'eray~\cite{F}, given a family of polynomials $K_\mu$, we can inductively define another family $\tilde{K}_{\mu}$ by
\[K_\mu=\sum_{k=1}^{l(\mu)} (-1)^{l(\mu)-k} \sum_{(\nu_1,\ldots, \nu_k)} \prod_{i=1}^k  \tilde{K}_{\nu_i},\]
where the second sum is taken over all decompositions of the $l(\mu)$ parts of $\mu$ into $k$ disjoint partitions $(\nu_1,\cdots,\nu_k)$. For instance one has $K_r=\tilde{K}_{r}$, $K_{r,s}=K_rK_s-\tilde{K}_{r,s}$ and $K_{r,s,t}=K_rK_sK_t-K_r\tilde{K}_{s,t}-K_s\tilde{K}_{r,t}-K_t\tilde{K}_{r,s}+\tilde{K}_{r,s,t}$.

Our computations support the following positivity conjecture.
\begin{conj}
Denote $\beta= 1-\al$. There is a ``natural'' expression of the coefficients of $\tilde{K}_\mu$ as polynomials in $(\al,\beta)$ with $\mathrm{nonnegative}$ integer coefficients. 
\end{conj}

In particular it is the case for $K_r=\tilde{K}_r$. Such a ``natural'' method to express $\tilde{K}_\mu$ will be presented in the course of this paper.  

Writing $R_i$ instead of $R_i(\la;\al)$ for clarity of display, the values obtained for $K_r, r\le 6$ are as follows:
\begin{equation*}
\begin{split}
K_{2} = & \ \al^2 R_3+\al\beta R_2,\\
K_{3} = & \ \al^3 R_4+3\al^2\beta R_3
+(\al^2+2\al\beta^2)R_2, \\
K_{4} = & \ 
\al^4 R_5+\al^3\beta (6R_4+R_2^2)
+(5\al^3+11\al^2\beta^2)R_3+(7\al^2\beta+6\al\beta^3)R_2,\\
K_{5} = & \
\al^5 R_6+\al^4\beta (10R_5+5 R_3R_2)
+\al^4(15R_4+5R_2^2)+\al^3\beta^2(35R_4+10R_2^2)
 \\ & +(55\al^3\beta+50\al^2\beta^3)R_3
+(8\al^3+46\al^2\beta^2+24\al\beta^4)R_2,\\
K_6 = & \ \al^6 R_7+\al^5\beta (15 R_6+9R_4R_2+6 R_3^2+R_2^3) 
+\al^5(35R_5+35R_3R_2)
\\ & +\al^4\beta^2(85R_5+73R_3R_2) +\al^4\beta(238R_4+96R_2^2)
 +\al^3\beta^3(225R_4+84R_2^2)
\\ & +(84\al^4+505\al^3\beta^2+274\al^2\beta^4)R_3
+(144\al^3\beta+326\al^2\beta^3+120\al\beta^5)R_2.
\end{split}
\end{equation*}
The values of $\tilde{K}_{\mu}, |\mu|\le 6$ are as follows:
\begin{equation*}
\begin{split}
\tilde{K}_{2,2} = & \  
\al^3 (4R_4+2R_2^2)+10\al^2\beta R_3+(2\al^2+6\al\beta^2)R_2,\\
\tilde{K}_{3,2} = & \  
\al^4(6R_5+6R_3R_2)+\al^3\beta (30R_4+12R_2^2)
+(18\al^3+48\al^2\beta^2)R_3
\\ & +(24\al^2\beta+24\al\beta^3)R_2,\\
\tilde{K}_{4,2} = & \ \al^5 (8R_6+8R_4R_2+4R_3^2)+\al^4\beta (68R_5+72 R_3R_2) +\al^4(80R_4+40R_2^2)\\ &
+\al^3\beta^2(208R_4+88R_2^2) +(268\al^3\beta+268\al^2\beta^3)R_3
+(32\al^3+212\al^2\beta^2+120\al\beta^4)R_2,\\
\tilde{K}_{3,3}= &  \ \al^5 (9R_6+9R_4R_2+9R_3^2+3R_2^3)+\al^4\beta (72R_5+81 R_3R_2) +\al^4(75R_4+27R_2^2)\\ &
+\al^3\beta^2(213R_4+90R_2^2) +(261\al^3\beta+270\al^2\beta^3)R_3
+(36\al^3+210\al^2\beta^2+120\al\beta^4)R_2,\\
\tilde{K}_{2,2,2} = & \ \al^4 (40R_5+64R_3R_2)+\al^3\beta (176R_4+96R_2^2)+(80\al^3+256\al^2\beta^2)R_3
\\ & +(104\al^2\beta+120\al\beta^3)R_2.
\end{split}
\end{equation*}

Tables giving $K_\mu$ for $|\mu|-l(\mu)\le 8$, $K_r$ for $r\le 20$ and $\tilde{K}_{rs}$ for $r+s\le 18$ are available on a web page~\cite{W}. Of course for $\al=1$ we recover the classical framework of Kerov-Biane-F\'eray mentioned at the beginning.

The purpose of this paper is twofold:
\begin{enumerate}
\item[(i)] to present a method for computing the expansion of $\vartheta^\la_\mu(\al)$ in terms of the free cumulants $R_i$ and other related bases,
\item[(ii)] to give precise conjectures about the polynomials $K_r$ and the nonnegativity of their coefficients. 
\end{enumerate}
Here we shall only indicate that $K_r$ is conjectured to write as
\[K_r = \sum_{\begin{subarray}{c}0\le j \le i \le r-1 \\ 2i-j \le r-1\end{subarray}}
 \al^{r-i} \beta^j \sum_{|\rho|=r-2i+j+1}
a_{ij}(\rho)  \, \prod_{k\ge 1} R_{\rho_k},\]
which may be easily checked on the previous examples. The coefficients $a_{ij}(\rho)$ will be studied below. 
 
Our results make highly plausible the existence of some combinatorial pattern, generalizing the one given by F\'eray~\cite{F} when $\alpha=1$. We conjecture the existence of such an $(\alpha,\beta)$-scheme, underlying the classical theory of Jack polynomials, where $\beta$ would play a role similar to $\al$.

It is also a highly remarkable fact that our method, results and conjectures remain entirely valid when $\beta$ is an independent parameter distinct from $1-\al$. This suggests the existence of a two-parameter $(\alpha,\beta)$-generalization of Jack polynomials.

The paper is organized as follows. In Section 2 we fix our notations about partitions and (shifted) symmetric functions. Section 3 is devoted to the expansion of Jack polynomials in terms of power sums. Section 4 recalls many useful formulas for ``Lagrange involution''. In Section 5 we define the anisotropic Young diagram of a partition and its free cumulants. In Section 6 we study the associated transition measure. Our method to compute $\vartheta^\la_\mu(\al)$ in terms of the free cumulants is presented in Sections 7 to 9. Many conjectures about the polynomials $K_r$ are given in Sections 10-11. Two other bases related with free cumulants are introduced  in Sections 12-13, and the corresponding expansions of $K_r$ are considered. Section 14 extends the previous conjectures to $\tilde{K}_\mu$. A two-parameter $(\alpha,\beta)$-generalization of our results is presented in Section 15.

\section{Generalities}

The standard references for symmetric functions and Jack polynomials are~\cite[Section 6.10]{Ma} and~\cite{S}. 

\subsection{Symmetric functions}

A partition $\la= (\la_1,...,\la_l)$
is a finite weakly decreasing
sequence of nonnegative integers, called parts. The number
$l(\la)$ of positive parts is called the length of
$\la$, and $|\la| = \sum_{i = 1}^{l(\la)} \la_i$
the weight of $\la$. 

For any integer $i\geq1$,
$m_i(\la) = \textrm{card} \{j: \la_j  = i\}$
is the multiplicity of the part $i$ in $\la$.  Clearly
$l(\la)=\sum_{i} m_i(\la)$ and
$|\la|=\sum_{i} im_i(\la)$.
We shall also write
$\la=(1^{m_1(\la)},2^{m_2(\la)},\ldots)$ and set
\[z_\la  = \prod_{i \ge  1} i^{m_i(\lambda)} m_i(\lambda) !, \qquad u_\la=l(\la)!/\prod_{i\ge 1} m_i(\la)!.\]
We identify $\la$ with its Ferrers diagram  
$\{ (i,j) : 1 \le i \le l(\la), 1 \le j \le {\la}_{i} \}$. We denote $\la^\prime$ the partition conjugate to $\la$. We have $m_i(\la^\prime)=\la_i-\la_{i+1}$, and $\la^\prime_i=\sum_{j \ge i} m_j(\la)$.

Let $A=\{a_1,a_2,a_3,\ldots\}$ be a (possibly infinite) set of independent indeterminates (an alphabet) and $\mathbb{S}[A]$ be the corresponding algebra of symmetric functions with coefficients in $\mathbf{Q}$.
The power sum symmetric functions
are defined by $p_{k}(A)=\sum_{i \ge 1} a_i^k$.
Elementary and complete symmetric functions $e_{k}(A)$ and $h_{k}(A)$
are defined by their generating functions
\[ E_t(A)=\prod_{i \ge 1} (1 +ta_i) =\sum_{k\geq0} t^k\, e_k(A), \qquad
H_t(A)=\prod_{i \ge 1}  \frac{1}{1-ta_i} = \sum_{k\geq0} t^k\, h_k(A) .\]

If $A$ is infinite, each of these three sets forms an algebraic basis of $\mathbb{S}[A]$, which can thus be viewed as an abstract algebra $\mathbb{S}$ generated by functions $e_{k}$, $h_{k}$ or $p_{k}$.
For any partition $\mu$, the symmetric functions $e_{\mu}$, $h_{\mu}$
and $p_{\mu}$ defined by
\[f_{\mu}=
\prod_{i=1}^{l(\mu)}f_{\mu_{i}}=\prod_{i\geq1}f_i^{m_{i}(\mu)},\]
where $f_i$ stands for $e_i$, $h_i$ or $p_i$ respectively,
form a linear basis of $\mathbb{S}$. Another classical
basis is formed by the monomial symmetric functions
$m_{\mu}$, defined as the sum of all distinct monomials
whose exponent is a permutation of $\mu$.

\subsection{Shifted symmetric functions}

Although the theory of symmetric functions goes back to the early 19th century, the notion of ``shifted symmetric'' functions is quite recent. We refer to~\cite{Ok1,Ok2,Ok3} and to other references given there.

Let $\mathbf{Q}(\alpha)$ be the field of rational functions in
some indeterminate $\alpha$ (which may be considered as a positive real number). A polynomial in $N$ indeterminates $x=(x_1,\ldots,x_N)$ with coefficients in $\mathbf{Q}(\alpha)$ is said to be ``shifted symmetric'' if it is symmetric in the $N$ ``shifted variables'' 
$x_i-i/\alpha$. 

Dealing with an infinite set of indeterminates $x=\{x_1,x_2,\ldots\}$, in analogy with symmetric functions, 
a ``shifted symmetric function'' $f$ is a family $\{f_i, i\ge 1\}$ such that  $f_i$ is shifted symmetric in $(x_1,x_2,\ldots,x_i)$, with the stability property $f_{i+1}(x_1,x_2,\ldots,x_i,0)=f_i(x_1,x_2,\ldots,x_i)$.

This defines $\mathcal{S}^{\ast}$, the shifted symmetric algebra with coefficients in $\mathbf{Q}(\alpha)$. An element $f\in \mathcal{S}^{\ast}$ may be evaluated at any sequence 
$(x_1,x_2,\ldots)$ with finitely many non zero terms, hence at any partition $\la$. Moreover by analyticity, $f$ is entirely determined by its restriction $f(\la)$ to partitions. Therefore $\mathcal{S}^{\ast}$ is usually considered as a function 
algebra on the set of partitions.

An algebraic basis of $\mathcal{S}^{\ast}$ is obtained as follows~\cite[Section 7, p. 3467]{La1}. Given a partition $\la$, the ``$\al$-content'' of any node $(i,j) \in \la$ is defined as $j-1-(i-1)/\al$. Consider the finite alphabet of $\al$-contents
\[C_\la=\{j-1-(i-1)/\al,\,(i,j) \in \la\}.\]
The symmetric algebra $\mathbb{S}[C_\la]$ is generated by the power sums
\begin{equation}
p_k(C_\la) = \sum_{(i,j) \in \la} (j-1-(i-1)/\al)^k\qquad (k\ge 1).
\end{equation}

It is known~\cite[Lemma 7.1, p. 3467]{La1} that the quantities $p_k(C_\la)$ are shifted symmetric functions of $\la$. As a direct consequence, the shifted symmetric algebra $\mathcal{S}^{\ast}$ is algebraically generated by the functions $\{p_k(C_\la), k\ge 1\}$ together with $p_0(C_\la)=|\la|$, the cardinal of the alphabet $C_\la$.

\section{Jack polynomials}

\subsection{Notations}

Let $\mathcal{S}=\mathbb{S}\otimes\mathbf{Q}(\alpha)$
be the algebra of symmetric functions with coefficients in $\mathbf{Q}(\alpha)$. The parameter $\al$ being kept fixed, for clarity of display, we shall omit its dependence in \textit{any notation} below.

The algebra $\mathcal{S}$ may be endowed with a scalar product 
$<\, , \,>$ for which we have two orthogonal bases, both indexed by partitions:
\begin{enumerate}
\item[(i)] the basis of power sum symmetric functions, with
\[<p_\la,p_\mu>=\delta_{\la \mu}\,\alpha^{l(\la)} z_{\la},\]
\item[(ii)] the basis of (suitably normalized) Jack symmetric functions, with
\[<J_\la,J_\mu>=\delta_{\la \mu}\,h_\la h^{\prime}_\la,\]
\end{enumerate}
and
\begin{equation*}
h_\la=\prod_{(i,j) \in \la} \left(\la^{\prime}_j-i+1+\al(\la_i-j)\right),\quad \quad
h^{\prime}_\la=\prod_{(i,j) \in \la}  \left(\la^{\prime}_j-i+\al(\la_i-j+1)\right).
\end{equation*}

We write $\theta^{\la}_{\rho}$ for the transition matrix between these two orthogonal bases, namely
\begin{equation}
J_{\la}=\sum_{|\rho|= |\la|} \theta^{\la}_{\rho} \,p_{\rho}.
\end{equation} 
Let $\mu$ be a partition with weight $r=|\mu|\le |\la|=n$ and $(\mu,1^{n-r})$ the partition obtained by adding $n-r$ parts $1$. We define 
\[\vartheta^{\la}_{\mu}=z_{\mu}\, \theta^{\la}_{\mu,1^{n-r}}.\]
It is known~\cite[Proposition 2]{La4} that $\vartheta^{\la}_{\mu}$ is a shifted symmetric function of $\la$. 

By restriction of Jack symmetric functions to a finite alphabet $x=(x_1,\ldots,x_N)$ we obtain Jack polynomials, which are eigenfunctions of the differential operator
\[D_2=\sum_{i=1}^N x_i^2\, \frac{\partial^2}{\partial x_i^2}
+\frac{2}{\al} \sum_{\begin{subarray}{c}i,j=1\\i\neq j\end{subarray}}^N 
\frac{x_i^2}{x_i-x_j}\frac{\partial}{\partial x_i}.\]
We have~\cite[p. 84]{S} 
\begin{equation}
D_2 J_\la=2\big(p_1(C_\la)+|\la|(N-1)/\al\big)J_\la.
\end{equation} 

\subsection{Pieri formula}

For any partition $\la$ and any integer $1 \le i \le l(\la)+1$, we denote by $\la^{(i)}$ the partition $\nu$ (if it exists) obtained by adding a node on the row $i$ of the diagram of $\la$, i.e. $\nu_j=\la_j$ for $j\neq i$ and $\nu_i=\la_i +1$. 

Given a partition $\mu$, for any integer $r\ge 1$ we denote by $\mu \cup r$ the partition obtained by adding a part $r$, and $\mu \setminus r$ the partition (if it exists) obtained by subtracting $r$. We write $\mu_{\downarrow (r)}=\mu \setminus r\cup (r-1)$ and $\mu_{\uparrow (r)}=\mu \setminus r\cup (r+1)$. 

Jack symmetric functions satisfy the following generalization of the Pieri formula~\cite{Ma,S}: 
\begin{equation}
p_1 \,J_\la =\sum_{i=1}^{l(\la)+1} 
c_i(\la) \, J_{\la^{(i)}}.
\end{equation}
The Pieri coefficients $c_i(\la)$ have the following analytic 
expression~\cite{La0}:
\[c_i(\la) = \frac {1}{\alpha \la_i+l(\la)-i+2}
\prod_{\begin{subarray}{c}j=1 \\ j \neq i\end{subarray}}^{l(\la)+1} 
\frac{\alpha(\la_i-\la_j)+j-i+1}
{\alpha(\la_i-\la_j)+j-i}.\]
The differential operator
\[E_2=\sum_{i=1}^N x_i^2\, \frac{\partial}{\partial x_i}\]
is independent of $N$. Denoting $p_1$ the multiplication operator $f\rightarrow p_1f$, we have easily
\begin{equation*}
E_2=\frac{1}{2} [D_2,p_1]-\frac{1}{\al}(N-1)p_1.
\end{equation*}
Due to (3.2) this implies
\begin{equation}
E_2 J_\la=\sum_{i=1}^{l(\la)+1} 
c_i(\la) \,\big(\la_i-(i-1)/\al\big)\, J_{\la^{(i)}}.
\end{equation}

We shall need the following equivalent form of equations (3.3)-(3.4). Expanding Jack polynomials in terms of power sums by (3.1), we obtain
\begin{equation}
\begin{split}
\theta^\la_{\rho_{\downarrow (1)}}&=\sum_{i=1}^{l(\la)+1} 
c_i(\la) \, \theta^{\la^{(i)}}_\rho,\\
\sum _{r\ge 1}
r\big(m_r(\rho)+1\big)\theta^\la_{\rho_{\downarrow (r+1)}}&=
\sum_{i=1}^{l(\la)+1} 
c_i(\la) \, \,\big(\la_i-(i-1)/\al\big)\,\theta^{\la^{(i)}}_\rho.
\end{split}
\end{equation}
The first relation is obvious. The second is a direct consequence of
\[E_2p_\rho=\sum_{r\ge 1} rm_r(\rho)\,p_{\rho_{\uparrow (r)}},\]
and the fact that if $\sigma=\rho_{\uparrow (r)}$, then $\rho=\sigma_{\downarrow (r+1)}$.
\begin{theo}~\label{th1}
Up to the multiplicative constant $\theta_{(1)}^{(1)}=1$, the system (3.5) totally determines the family of Jack polynomials $J_\la$.
\end{theo}
\begin{proof}
Starting from the initial case $\la=\rho=(1)$, we compute any $\theta^{\la}_{\rho}$ by two inductions, firstly on the length $l(\la)$, secondly on the value of the last part $\la_{l(\la)}$. 

Our inductive assumption is the following: given a length $d\ge 1$, the quantity $\theta^{\la}_\rho$ is known for any $\la=(\la_1,\ldots,\la_{d-1},\la_d)$ with $\la_d\le u < \la_{d-1}$ and any $\rho$ with $|\rho|=|\la|$. 

If we write equations (3.5) for $\la=(\la_1,\ldots,\la_{d-1},u)$ and $|\rho|=|\la|+1$, we obtain
\begin{equation*}
\begin{split}
c_d(\la) \,\theta^{(\la_1,\ldots,\la_{d-1},u+1)}_\rho+
c_{d+1}(\la) \,\theta^{(\la_1,\ldots,\la_{d-1},u,1)}_\rho
&= [\cdots],\\
c_d(\la) \, \big(u-(d-1)/\al\big)\,\theta^{(\la_1,\ldots,\la_{d-1},u+1)}_\rho+
c_{d+1}(\la) \, \big(-d/\al\big)\,\theta^{(\la_1,\ldots,\la_{d-1},u,1)}_\rho
&= [\cdots],
\end{split}
\end{equation*}
where both right-hand sides $[\cdots]$ are known in view of the inductive hypothesis. Actually they are linear combinations firstly of quantities $\theta^{\la}_{\rho_{\downarrow (r)}}$ and secondly of quantities $\theta^{\la^{(i)}}_{\rho}$ with $\la^{(i)}=(\la_1,\ldots,\la_i+1,\ldots,\la_{d-1},u)$. 

This linear system has a discriminant equal to $c_d(\la)c_{d+1}(\la) (u+1/\al)$, which is $\neq 0$ due to $u < \la_{d-1}$. Hence the system has two unique solutions. In other words $\theta^{\la}_\rho$ is known for $\la=(\la_1,\ldots,\la_{d-1},u+1)$ and $\la=(\la_1,\ldots,\la_{d-1},u,1)$. 

This proves the second inductive property (on the value of $\la_{d}$), and initiates a new induction for the length $l(\la)=d+1$ with $\la_{d+1}=1$. Hence the statement.
\end{proof}

Writing $\rho=(\mu,1^{n-r})$ the system (3.5) can be easily translated in terms of the quantities $\vartheta^{\la}_{\mu}$. For any partition $\mu$ having no part $1$, we obtain
\begin{equation}
\begin{split}
\sum_{i=1}^{l(\la)+1} 
c_i(\la) \, \vartheta^{\la^{(i)}}_\mu&=\vartheta^\la_\mu,\\
\sum_{i=1}^{l(\la)+1} 
c_i(\la)  \,\big(\la_i-(i-1)/\al\big)\,\vartheta^{\la^{(i)}}_\mu&=2 (|\la|-|\mu|+2)\,m_2(\mu)\,\vartheta^\la_{{\mu \setminus 2}}+
\sum _{r\ge 3}r\,m_r(\mu)\,\vartheta^\la_{\mu_{\downarrow (r)}}.
\end{split}
\end{equation}
In the right-hand side the case $r=2$ must be distinguished because the partition $\mu_{\downarrow (2)}$ has one part $1$.

The system (3.6) is our starting point for the computation of $\vartheta^{\la}_{\mu}$.

\section{Lagrange involution}

\subsection{Multiplication of alphabets}

Given two alphabets $A$ and $B$, their sum $A+B$ 
and their difference $A-B$ are defined by the generating functions
\begin{equation*}
H_t(A+B)=H_t(A)\,H_t(B), \quad\quad 
H_t(A-B)=H_t(A)/{H_t(B)}.
\end{equation*}
Equivalently
\begin{equation}
h_n(A+B)=\sum_{k=0}^n h_k(A)\,h_{n-k}(B).
\end{equation}

Given a \textit{positive integer} $x$, the  multiple $xA$ is the alphabet formed by $x$ copies of $A$. Obviously we have
\[p_n(xA)=x p_n(A),\quad\quad
p_\mu(xA)=x^{l(\mu)} p_\mu(A).\]
And the generating functions of this alphabet
\begin{equation}
E_t(xA)=E_t(A)^x,\qquad H_t(xA)=H_t(A)^x
\end{equation}
are the product of $x$ copies of $E_t(A)$ or $H_t(A)$.

Both properties are chosen as a \textit{definition} of the ``multiple'' $xA$ when $x$ is any \textit{real number}. The symmetric functions of $xA$ are defined accordingly.

In this framework we have two important ``Cauchy formulas''~\cite[Section 1.6]{Las}, see also~\cite[pp. 62-65]{Ma} or~\cite[p. 222]{La5}. Firstly
\begin{equation}
\begin{split}
h_n(xA)&=\sum_{|{\mu }| = n} 
 z_\mu^{-1} x^{l(\mu)} \, p_{\mu} (A),\\
e_{n} (xA)&= \sum_{|{\mu }| = n} 
(-1)^{n-l(\mu)}  \, z_\mu^{-1} x^{l(\mu)} \, p_{\mu} (A).
\end{split}
\end{equation}
Secondly
\begin{equation}
\begin{split}
h_n(xA)&=\sum_{|\mu|=n} \binom{x}{l(\mu)} u_\mu h_\mu(A)\\
&=(-1)^{n} \sum_{|\mu|=n} \binom{-x}{l(\mu)} u_\mu e_\mu(A).
\end{split}
\end{equation}
Equivalently
\begin{equation}
\begin{split}
e_n(xA)&=\sum_{|\mu|=n} \binom{x}{l(\mu)} u_\mu e_\mu(A)\\
&=(-1)^{n}\sum_{|\mu|=n} \binom{-x}{l(\mu)} u_\mu h_\mu(A).
\end{split}
\end{equation}
In this paper the binomial notation $\binom{x}{k}$ will always stand for $x(x-1)\cdots(x-k+1)/k!$. Multiplication of alphabets and Cauchy formulas are better understood in the language of $\lambda$-rings. See~\cite[Chapter 2]{Las} (or~\cite[Section 3]{La5} for an elementary account) and also Section 13 below.

\subsection{Lagrange formula}
Given an alphabet $A$, an involution $f \rightarrow f^{*}$ can be defined on $\mathbb{S}[A]$ as follows (\cite[Example 1.2.24, p. 35]{Ma}, \cite[Section 2.4]{Las}). Let
\[u=tH_t(A)=\sum_{k\geq0} t^{k+1}\, h_k(A).\]
Then $t$ can be expressed as a power series in $u$, its compositional inverse, namely
\[t=uH_u^{*}(A)=\sum_{k\geq0} u^{k+1}\, h^{*}_k(A).\]
The map $h_k(A) \rightarrow h_k^{*}(A)$ extends to an involution of $\mathbb{S}[A]$, called ``Lagrange involution''.

The Lagrange inversion formula asserts, in one of its equivalent forms~\cite{GJ}, that
\begin{equation}
[u^n]H_u^{*}(A)^k=\frac{k}{n+k}[t^n]H_t(A)^{-n-k},
\end{equation}
where $[t^n]F_t$ denotes the coefficient of $t^n$ in the formal series $F_t$. This fundamental result has many consequences, listed below.

\subsection{Consequences}

Firstly omitting to mention the alphabet $A$, we have
\begin{equation}
\begin{split}
(n+1)h_n^{*}&=\sum_{|\mu|=n} (-n-1)^{l(\mu)} z_\mu^{-1} p_\mu\\
&=\sum_{|\mu|=n}  \binom{-n-1}{l(\mu)} u_\mu h_\mu\\
&=(-1)^n\sum_{|\mu|=n} \binom{n+1}{l(\mu)} u_\mu e_\mu.
\end{split}
\end{equation}
Secondly
\begin{equation}
\begin{split}
(n-1)e_n^{*}&=(-1)^{n-1}\sum_{|\mu|=n} (1-n)^{l(\mu)} z_\mu^{-1} p_\mu\\
&=-\sum_{|\mu|=n} \binom{n-1}{l(\mu)} u_\mu e_\mu\\
&=(-1)^{n-1}\sum_{|\mu|=n} \binom{-n+1}{l(\mu)} u_\mu h_\mu.
\end{split}
\end{equation}
Finally we have
\begin{equation}
p_n^{*}=\sum_{|\mu|=n} (-n)^{l(\mu)} z_\mu^{-1} p_\mu.
\end{equation}

This formula exhibits an important connection between Lagrange involution and multiplication of alphabets. 
By the classical Cauchy formulas 
\[h_n=\sum_{|\mu|=n} z_\mu^{-1} p_\mu,\quad\quad
e_n=\sum_{|\mu|=n} (-1)^{n-l(\mu)}z_\mu^{-1} p_\mu,\]
it reads
\[p_n^{*}(A)=\sum_{|\mu|=n} z_\mu^{-1} p_\mu(-nA)=h_n(-nA)= (-1)^n e_n(nA).\]
Similarly from (4.7) and (4.8) we obtain
\begin{equation}
\begin{split}
(n+1)h_n^{*}(A)&=h_n((-n-1)A)=(-1)^ne_n((n+1)A),\\
(n-1)e_n^{*}(A)&=(-1)^{n-1}h_n((1-n)A)=-e_n((n-1)A).
\end{split}
\end{equation}

These properties may be extended as follows. For any real number $x$ and any alphabet $A$, Lagrange formula (4.6) yields
\begin{equation}
\begin{split}
(n+x)h_n(xA)
&=x \sum_{|\mu|=n} \binom{-n-x}{l(\mu)} u_\mu h^*_\mu(A)\\
&=(-1)^n x \sum_{|\mu|=n} \binom{n+x}{l(\mu)} u_\mu e^*_\mu(A).
\end{split}
\end{equation}
Similarly
\begin{equation}
\begin{split}
(n-x)e_n(xA)
&=-x\sum_{|\mu|=n} \binom{n-x}{l(\mu)} u_\mu e^*_\mu(A)\\
&=(-1)^{n-1}x\sum_{|\mu|=n} \binom{-n+x}{l(\mu)} u_\mu h^*_\mu(A).
\end{split}
\end{equation}
From which we obtain
\begin{equation}
\begin{split}
(n+x)h_n^{*}(xA)&=x\,h_n((-n-x)A)=(-1)^nx\,e_n((n+x)A),\\
(n-x)e_n^{*}(xA)&=(-1)^{n-1}xh_n((x-n)A)=-x\,e_n((n-x)A).
\end{split}
\end{equation}

\begin{proof} Relations (4.7)-(4.10) are obtained from (4.11)-(4.13) by specializing $x=1$. The proofs of the latter are strictly parallel consequences of the Lagrange inversion formula (4.6). As an example let us prove the second part of (4.11). We have
\begin{equation*}
\begin{split}
h_n^*(xA)&=[u^n]H_u^*(A)^x=\frac{x}{n+x}[t^n]H_t(A)^{-n-x}\\
&=\frac{x}{n+x}[t^n]\Big(1+\tilde{E}_{-t}(A)\Big)^{n+x}\\
&=\frac{x}{n+x}[t^n]\Big(\sum_{d\ge 0}\binom{n+x}{d}
\tilde{E}_{-t}(A)^d\Big),
\end{split}
\end{equation*}
with $\tilde{E}_{-t}=\sum_{k\geq1} (-t)^k\, e_k$. But
\[[t^n]\tilde{E}_{-t}(A)^d=(-1)^n\sum_{\begin{subarray}{c} |\mu|=n\\l(\mu)=d\end{subarray}} u_\mu e_\mu(A).\]
\end{proof}

\subsection{Some identities}

For any partition $\rho$ define
\[v_\rho=\prod_{i\ge 1} (i-1)^{m_i(\rho)},\quad \quad w_\rho=v_\rho \sum_{i\ge 2} \frac{im_i(\rho)}{i-1}.\]
For any alphabet $A$ and any real number $z$, we have the remarkable identities
\begin{equation}
\begin{split}
&\sum_{k=0}^n \frac{z}{z+k} h_{k}(-(z+k)A) h_{n-k}((z+k-1)A)=(-1)^n e_n(A),\\
&\sum_{k=0}^n \frac{1}{z+k} h_{k}(-(z+k)A) h_{n-k}((z+k)A)=0,\\
&\sum_{k=0}^n h_{k}(-(z+k)A) h_{n-k}((z+k+1)A) =(-1)^n \sum_{|\rho|=n} v_\rho u_\rho e_\rho(A), \\
&\sum_{k=0}^n h_{k}(-(z+k)A) h_{n-k}((z+k)A) =(-1)^{n} \sum_{|\rho|=n} w_\rho \frac{u_\rho}{l(\rho)} e_\rho(A). 
\end{split}
\end{equation}
The assumption $e_1(A)=0$ is needed for the last one. Observe that the right-hand sides are independent of $z$. 

Although these identities have deep connections with $\la$-ring theory, an elementary proof may be given as follows. Expand each $h_m(xA)$ in terms of the basis $e_\rho(A)$ by the Cauchy formula (4.4). Then we are left with a relation involving binomial coefficients. For instance the third identity amounts to write
\[\sum_{\mu\cup\nu=\rho}u_\mu u_\nu \binom{|\mu|t+z}{l(\mu)} \binom{-|\mu|t-z-1}{l(\nu)} =u_\rho
\prod_{i\ge 1} (ti-1)^{m_i(\rho)},\]
at $t=1$, the sum taken over all decompositions of $\rho$ into two partitions (possibly empty). 

This result is a consequence of a more general property, obtained by replacing the parts of $\rho$ by any alphabet $A$, namely
\[\sum_{X\subset A}(-1)^{n-\mathrm{card} X} \prod_{i=1}^{n}(|X|t-\mathrm{card} X+z+i)=n! \,\prod_{a\in A}(ta-1),\]
with $n=\mathrm{card} A$ and $|X|=\sum_{a\in X}a$. Omitting details, the other identities are implied by
\begin{equation*}
\begin{split}
&\sum_{X\subset A}(-1)^{n-\mathrm{card} X} |X| \prod_{i=1}^{n-2}(|X|t-\mathrm{card} X+z+i)=0\hspace{1,5cm} (n \neq 1),\\
&\sum_{X\subset A}(-1)^{n-\mathrm{card} X} \prod_{i=1}^{n-m}(|X|t-\mathrm{card} X+z+i)=0\hspace{2cm} (m=1,2),\\
&\sum_{X\subset A}(-1)^{n-\mathrm{card} X} |X|\, \prod_{i=1}^{n-1}(|X|t-\mathrm{card} X+z+i)=(n-1)! \frac{d}{dt} \prod_{a\in A}(ta-1). 
\end{split}
\end{equation*}

Four similar identities may be obtained by applying Lagrange involution to (4.14) and using (4.13), namely
\begin{equation}
\begin{split}
&\sum_{k=0}^n \frac{z+k-1}{z+n-1} h_{k}(zA) h_{n-k}(-(z+n-1)A)=(-1)^n e_n^{*}(A),\\
&\sum_{k=0}^n (z+k) h_{k}(zA) h_{n-k}(-(z+n)A)=0,\\
&\sum_{k=0}^n \frac{(z+k)(z+k+1)}{z(z+n+1)}h_{k}(zA) h_{n-k}(-(z+n+1)A) =(-1)^n \sum_{|\rho|=n} v_\rho u_\rho e_\rho^{*}(A), \\
&\sum_{k=0}^n \frac{(z+k)^2}{z(z+n)}h_{k}(zA) h_{n-k}(-(z+n)A) =(-1)^{n} \sum_{|\rho|=n} w_\rho \frac{u_\rho}{l(\rho)} e_\rho^{*}(A), 
\end{split}
\end{equation}
with $e_1(A)=0$ assumed for the last one.

\section{Free cumulants}

The following notions are due to Kerov~\cite{K1,K3}. Two increasing sequences $x_1,\ldots,x_{d-1},x_d$ and $y_1,\ldots,y_{d-1}$ with $x_1<y_1<x_2<\cdots<x_{d-1}<y_{d-1}<x_d$ are said to form a pair of interlacing sequences. The center of this pair is $\sum_i x_i-\sum_i y_i$.

Given a partition $\la$, its Young diagram $Y_\la \subset \mathbf{R}^2$ is defined as the collection of unit boxes centered on the nodes
$\{ (j-1/2,i-1/2), (i,j)\in \la \}$. 
An outside corner of $Y_\la$ is defined as the north-east corner of a corner box. An inside corner is the south-west corner of a corner box of the complement of $Y_\la$ in $\mathbf{R}^2$.

On $\mathbf{R}^2$ define the content function $c(u,v)=u-v$. Then it is easily checked that a pair of interlacing sequences is formed by the contents $x_1,\ldots,x_d$ of inside corners of $Y_\la$, and the contents $y_1,\ldots, y_{d-1}$ of its outside corners. This pair has center $0$, and we have $x_1=-l(\la)$ and $x_d=\la_1$. 

The collection of $1\times 1/\al$  boxes obtained from $Y_\la$ by the dilation $(u,v) \rightarrow (u,v/\al)$ is called the \textit{anisotropic} Young diagram $Y_\la(\al)$. Its corners define similarly two interlacing sequences $x_1(\al),\ldots,x_d(\al)$ and $y_1(\al),\ldots, y_{d-1}(\al)$. Defining the $\al$-content function by $c(u,v)=u-v/\al$, these sequences may also be understood as the $\al$-contents of inside and outside corners of $Y_\la$.

For instance if $\la=(4,3,3,3,1)$, the $x$ and $y$ sequences are respectively $(-5,-3,2,4)$ and $(-4,-1,3)$. The $x(\al)$ and $y(\al)$ sequences are respectively $(-5/\al,-4/\al+1,-1/\al+3,4)$ and $(-5/\al+1,-4/\al+3,-1/\al+4)$.

We have $x_1(\al)=-l(\la)/\al$ and $x_d(\al)=\la_1$. Moreover we may assume $d=\la_1+1$ and for $1\le k \le \la_1=d-1$ we may write
\begin{equation}
x_k(\al)=k-1-\la^\prime_k/\al,\quad \quad
y_k(\al)=k-\la^\prime_k/\al,
\end{equation}
\textit{provided we make the convention} that $x_i$ and $y_{i-1}$ should be omitted whenever $x_i=y_{i-1}$. Clearly this pair has center $0$.

For instance when $\la=(4,3,3,3,1)$, we have $\la^\prime=(5,4,4,1)$ and with this convention, the $x(\al)$ and $y(\al)$ sequences are respectively $(-5/\al,-4/\al+1,-4/\al+2,-1/\al+3,4)$ and $(-5/\al+1,-4/\al+2,-4/\al+3,-1/\al+4)$.

The ``transition measure'' of $Y_\la(\al)$ is a measure $\omega_\la$ on the real line, supported on the finite set $\{x_1(\al),\ldots,x_d(\al)\}$. It is uniquely defined by its moment generating series
\begin{equation*}
\begin{split}
\mathcal{M}_\la(z)
&=z^{-1}+\sum_{k\ge 1}M_k(\la) z^{-k-1}\\
&= \frac{1}{z-x_d(\al)}\prod_{i=1}^{d-1}\frac{z-y_i(\al)}{z-x_i(\al)}.
\end{split}
\end{equation*}
It is known~\cite{V,Sp} that the free cumulants $\{R_k(\la), k\ge 1\}$ of $\omega_\la$ are generated by the compositional inverse
\[\mathcal{R}_\la(u)=\mathcal{M}_\la^{(-1)}(u) = u^{-1} + \sum_{k\ge1} R_k(\la) u^{k-1}.\] 
The Boolean cumulants $\{B_k(\la), k\ge 1\}$ of $\omega_\la$ are generated by the inverse series
\[\mathcal{B}_\la(z)=(\mathcal{M}_\la(z))^{-1}=z-\sum_{k\ge 1}B_k(\la) z^{1-k}.\]
Observe that $M_1(\la)=B_1(\la)=R_1(\la)=0$ and  $M_2(\la)=B_2(\la)=R_2(\la)=|\la|/\al$.

Defining the ``inside'' and ``outside'' alphabets 
\[I_\la=\{x_1(\al),\ldots,x_d(\al)\},\quad\quad 
O_\la=\{y_1(\al),\ldots, y_{d-1}(\al)\},\]
and their difference $A_\la=I_\la-O_\la$, we then have
\begin{equation*}
\begin{split}
\mathcal{M}_\la(z)&=z^{-1}H_{1/z}(A_\la),\\
\mathcal{B}_\la(z)&=z\big(H_{1/z}(A_\la)\big)^{-1}=zE_{-1/z}(A_\la),\\
\mathcal{R}_\la(u)&=u^{-1}\big(H_{u}^*(A_\la)\big)^{-1}=
u^{-1}E_{-u}^*(A_\la).
\end{split}
\end{equation*}
Equivalently 
\[M_k(\la)= h_k(A_\la),\quad \quad
B_k(\la)= (-1)^{k-1}e_k(A_\la),\quad \quad
R_k(\la)=(-1)^ke_k^*(A_\la).\] 

We may specialize the results of Section 4 at $A=A_\la$. From (4.5) et (4.8) we obtain
\begin{equation}
\begin{split}
B_n(\la)
&=- \sum_{|\mu|=n} (-1)^{l(\mu)} u_\mu M_\mu(\la),\\
(n-1)B_n(\la)
&=\sum_{|\mu|=n} \binom{n-1}{l(\mu)} u_\mu R_\mu(\la).
\end{split}
\end{equation}
From (4.4) et (4.7) we get
\begin{equation}
\begin{split}
M_n(\la)
&= \sum_{|\mu|=n}  u_\mu B_\mu(\la),\\
(n+1)M_n(\la)&=\sum_{|\mu|=n} \binom{n+1}{l(\mu)} u_\mu R_\mu(\la).
\end{split}
\end{equation}
Conversely (4.8) becomes
\begin{equation}
\begin{split}
(n-1)R_n(\la)&=-\sum_{|\mu|=n} (-1)^{l(\mu)}\binom{n-1}{l(\mu)} u_\mu B_\mu(\la),\\
&=-\sum_{|\mu|=n} (-1)^{l(\mu)}\binom{n+l(\mu)-2}{l(\mu)} u_\mu M_\mu(\la).
\end{split}
\end{equation}
From (4.4) and (4.11) we have also
\begin{equation}
\begin{split}
h_n(xA_\la)&=\sum_{|\mu|=n} \binom{x}{l(\mu)} u_\mu M_\mu(\la)\\
&=\frac{x}{n+x}\sum_{|\mu|=n} \binom{n+x}{l(\mu)} u_\mu R_\mu(\la).
\end{split}
\end{equation}
In particular for $x=1-n$, we obtain 
\begin{equation}
(1-n)R_n(\la)=h_n((1-n)A_\la).
\end{equation}
Finally the specialization of (4.15) at $A=A_\la$ and $z=1$ yields
\begin{equation}
\sum_{k=1}^{n}M_{k}(\la) \,
\sum_{|\mu|= n-k } \binom{-k}{l(\mu)} u_\mu R_\mu(\la)=R_n(\la),
\end{equation}
\begin{equation}
\sum_{k=1}^{n}M_{k-1}(\la) \,
\sum_{|\mu|= n-k } \binom{-k}{l(\mu)} u_\mu R_\mu(\la)=0,
\end{equation}
\begin{equation}
\sum_{k=2}^{n} (k-1) M_{k-2}(\la) \,
\sum_{|\mu|= n-k } \binom{-k}{l(\mu)} u_\mu R_\mu(\la)=
\sum_{|\rho|=n-2} \, v_\rho u_\rho R_\rho(\la),
\end{equation}
\begin{equation}
\sum_{k=2}^{n}  (k-1) M_{k-1}(\la) \,
\sum_{|\mu|= n-k }\binom{-k}{l(\mu)} u_\mu R_\mu(\la)=
\sum_{|\rho|=n-1} \, w_\rho \frac{u_\rho}{l(\rho)} R_\rho(\la).
\end{equation}

\section{The transition measure}

In the previous section, the transition measure $\omega_\la$ was defined by its moment generating series $\mathcal{M}_\la(z)$. Equivalently we can write
\[\omega_\la=\sum_{k=1}^d \sigma_k(\al)\, \delta_{x_k(\al)},\]
where $\delta_{u}$ is the Dirac measure at $u$, and the weights $\sigma_k(\al)$ are the ``transition probabilities''
\[\sigma_{k}(\al)=\prod_{i=1}^{k-1}\frac{x_k(\al)-y_i(\al)}{x_k(\al)-x_i(\al)}\prod_{j=k+1}^d\frac{x_k(\al)-y_{j-1}(\al)}{x_k(\al)-x_j(\al)}.\]
Both definitions are linked by the formula
\[\mathcal{M}_\la(z)=\sum_{k=1}^{d} 
 \frac{\sigma_k(\al)}{z-x_k(\al)}.\]
Of course multiplying both sides by $z$ and taking the limit $z\rightarrow \infty$, we obtain
\[\sum_{k=1}^{d} \sigma_k(\al)=1.\]

Now consider the Pieri coefficients $\{c_i(\la), 1\le i\le l(\la)+1\}$. If we restrict to a finite set of $N$
indeterminates $x=(x_1,\ldots,x_N)$, we have~\cite[Theorem 5.4]{S}
\[J_\la(1,\ldots,1)=\prod_{(i,j) \in \la} \left(N+\al(j-1)-i+1\right).\]
Writing the Pieri formula (3.3) at $(1,\ldots,1)$, and identifying coefficients of $N$, we easily obtain
\[\sum_{i=1}^{l(\la)+1} c_i(\la)=1.\]

Thus we are led to compare two discrete probability distributions $\{\sigma_k(\al),1\le k \le d\}$ and $\{c_i(\la), 1\le i\le l(\la)+1\}$. This has been done by Kerov in a rather different form~\cite[Lemma 7.2]{K3}. His following result is fundamental for our purpose since \textit{it connects Jack polynomials with the transition measure}.
\begin{theo}~\label{th2}
Given a partition $\la$, the weights
$\{\sigma_k(\al),1\le k \le d\}$ are the non-zero elements of $\{c_i(\la), 1\le i\le l(\la)+1\}$. More precisely we may write
\[\omega_\la=\sum_{i=1}^{l(\la)+1} c_i(\la)\, \delta_{\la_i-(i-1)/\al}.\]
Equivalently we have
\[\mathcal{M}_\la(z)=\sum_{i=1}^{l(\la)+1} 
 \frac{c_i(\la)}{z-\la_i+(i-1)/\al}.\]
\end{theo}
 
In particular $\sigma_1(\al)=c_{l(\la)+1}(\la)$ and $\sigma_d(\al)=c_1(\la)$. In this section we shall give another proof of this result.

For any partition $\la$, we have
the following natural generalization of the ``rising factorial'' in terms of $\al$-contents
\[(x)_{\la}=\prod_{(i,j) \in \la} \left(x+j-1-(i-1)/\alpha\right).\]
Given two indeterminates $x,y$ the quantity
\[\frac{(x+y+1)_{\la}}{(x+y)_{\la}}\, \frac{(x)_{\la}}{(x+1)_{\la}}\]
has been studied in~\cite[Section 5, p. 3463]{La1}. Its development in descending powers of $x$ has been given there (see~\cite{La3} for an important application).

It turns out that the moment generating series of the transition measure is a function of this type.

\begin{theo}~\label{th3}
We have
\begin{equation*}
\mathcal{M}_\la(z)= z^{-1} \,\frac{(-z-1/\al+1)_{\la}}{(-z-1/\al)_{\la}}\, \frac{(-z)_{\la}}{(-z+1)_{\la}}.
\end{equation*}
\end{theo}
\begin{proof}
Denote $\mathcal{N}_\la(z)$ the right-hand side. In view of (5.1) and the convention made there, it is enough to prove
\[\mathcal{N}_\la(z)=\frac{1}{z-x_d(\al)}\prod_{k=1}^{d-1}\frac{z-y_k(\al)}{z-x_k(\al)},\]
with $d=\la_1+1$, $x_d(\al)=\la_1$ and for $1\le k \le \la_1$,
\[x_k(\al)=k-1-\la^\prime_k/\al,\quad \quad
y_k(\al)=k-\la^\prime_k/\al.\]
We have clearly
\[\frac {(u+1)_\la}{(u)_\la}=\prod_{i=1}^{l(\la)}
\prod_{j=1}^{\la_i} \frac{u+j-(i-1)/\al}{u+j-1-(i-1)/\al}
=\prod_{i=1}^{l(\la)} \frac{u+\la_i-(i-1)/\al}{u-(i-1)/\al},\]
which yields
\[z\, \mathcal{N}_\la(z)=
\prod_{i=1}^{l(\la)} \frac{z-\la_i+i/\al}{z+i/\al}\,
\frac{z+(i-1)/\al}{z-\la_i+(i-1)/\al}.\]
On the right-hand side consider the product corresponding to a part $\la_i=r$. Since $\la^\prime_k=\sum_{i\ge k}m_i(\la)$, it writes
\[ \prod_{i=\la^\prime_{r+1}+1}^{\la^\prime_r} \frac{z-r+i/\al}{z+i/\al}\,
\frac{z+(i-1)/\al}{z-r+(i-1)/\al}=
\frac{z-r+\la^\prime_r/\al}{z-r+\la^\prime_{r+1}/\al}\, \frac{z+\la^\prime_{r+1}/\al}{z+\la^\prime_r/\al}.\]
Bringing all contributions together we obtain
\[z\, \mathcal{N}_\la(z)=\frac{z}{z-\la_1}\prod_{r=1}^{\la_1}
\frac{z-r+\la^\prime_r/\al}{z-r+1+\la^\prime_r/\al}.\]
\end{proof}

Equivalently if we write
\[\frac{(-z-1/\al+1)_{\la}}{(-z-1/\al)_{\la}}\, \frac{(-z)_{\la}}{(-z+1)_{\la}}=\sum_{k\ge 0} s_k(\la) z^{-k},\]
we have $M_k(\la)=s_k(\la)$. Thus we may apply the results of~\cite{La1}, which have two important consequences. Firstly we recover Kerov's result.
\begin{theo}~\label{th4}
For any integer $k\ge 0$ we have
\begin{equation*}
M_k(\la)= \sum_{i=1}^{l(\la)+1} 
c_i(\la) \,\big(\la_i-(i-1)/\al\big)^k.
\end{equation*}
Equivalently
\[\mathcal{M}_\la(z)=\sum_{i=1}^{l(\la)+1} 
 \frac{c_i(\la)}{z-\la_i+(i-1)/\al}.\]
\end{theo}
\begin{proof} Using ``Lagrange interpolation''~\cite[Section 7.8]{Las}, the first statement has been proved for $s_k(\la)$ in~\cite[Theorem 8.1, p. 3470]{La1}.
\end{proof}

Secondly $M_k(\la)$, originally defined as some \textit{rational} function of $\la$, is actually a shifted symmetric \textit{polynomial}. 

\begin{theo}~\label{th5}
The moments $\{M_k(\la), k\ge 2\}$, boolean cumulants $\{B_k(\la), k\ge 2\}$ and free cumulants $\{R_k(\la), k\ge 2\}$ are three algebraic bases of the shifted symmetric algebra $\mathcal{S}^{\ast}$. They are polynomials in $1/\al$.
\end{theo}
\begin{proof}
Due to (5.2)-(5.4), it is enough to prove both statements for the moments. But they follow from~\cite[Corollary 5.2, p. 3464]{La1} where we have proved
\[s_k(\la)=\sum_{\begin{subarray}{c}p,q,r \ge 0 \\ p+q+2r \le k\end{subarray}} 
\frac{1}{\al^r} \Big(1-\frac{1}{\al}\Big)^p\binom{p+q+r-1}{p}
\,\sum_{s=0}^{\mathrm{min}(r,k-2r-p)}
\binom{|\la|+r-1}{r-s} F_{k-2r-p,q,s}(\la).\]
Here $F_{p,q,r}(\la)$ is a shifted symmetric function explicitly known in terms of the basis (2.1) (it is also a polynomial in $1/\al$).
\end{proof}

\section{An example}

Our method for the computation of $\vartheta^\la_\mu$ will be best understood through its simplest example $\mu=(2)$.
\begin{itemize}
\item[(i)] Fixing the constant $\vartheta^\la_{0}=\theta^\la_{1^{|\la|}}=1$ and using $|\la|=\al R_2(\la)$, equations (3.6) become
\begin{equation*}
\begin{split}
\sum_{i=1}^{l(\la)+1} 
c_i(\la) \, \vartheta^{\la^{(i)}}_2&=\vartheta^\la_2,\\
\sum_{i=1}^{l(\la)+1} 
c_i(\la) \, \,\big(\la_i-(i-1)/\al\big)\,\vartheta^{\la^{(i)}}_2&=2 \al R_2(\la).
\end{split}
\end{equation*}
The solution being unique, we search for it in the form \[\vartheta^\la_2=AR_4(\la)+BR_2^2(\la)+CR_3(\la)+DR_2(\la).\] 

\item[(ii)] We compute the quantities $R_k(\la^{(i)})$ in terms of $R_k(\la)$. Here we assume (this will be proved in Theorem~\ref{th6} below) that
\begin{equation*}
\begin{split}
\al^3R_4(\la^{(i)})&=\al^3R_4(\la)+3\al^2x_i^2-3\al\beta x_i-3\al^2R_2(\la)+\beta^2-\al,\\
\al^2R_3(\la^{(i)})&=\al^2R_3(\la)+2\al x_i-\beta,\\
\al R_2(\la^{(i)})&=\al R_2(\la)+1,
\end{split}
\end{equation*}
with $\beta=1-\al$ and $x_i=\la_i-(i-1)/\al$. 

\item[(iii)] We insert these values in the system (3.6) and apply Theorem~\ref{th4}. We obtain
\begin{equation*}
\begin{split}
&3\al^2AM_2(\la)-3\al^2AR_2(\la)+A(\beta^2-\al)+2\al^2BR_2(\la)+\al B-\al\beta C+\al^2D=0,\\
&3\al^2AM_3(\la)+(-3\al\beta A+2\al^2C)M_2(\la)=2\al^4R_2(\la).
\end{split}
\end{equation*}

\item[(iv)] We express the moments in terms of the free cumulants by (5.3). Here we have only $M_2(\la)=R_2(\la)$ and $M_3(\la)=R_3(\la)$.

\item[(v)] We apply Theorem~\ref{th5} and identify the polynomials in the free cumulants on both sides. This yields
\begin{equation*}
\begin{split}
B=0,&\quad A(\beta^2-\al)-\al\beta C+\al^2D=0,\\
A=0,&\quad 2\al^2C=2\al^4,
\end{split}
\end{equation*}
that is $A=B=0, C=\al^2, D=\al\beta$. Hence $\vartheta^\la_2=\al^2 R_3(\la)+\al\beta R_2(\la)$.
\end{itemize}

Observe that after (ii) we never used the fact that $\al+\beta=1$. For steps (iii) to (v), $\al$ and $\beta$ might be independent parameters. 

\section{Adding a node}

As shown above, we must first obtain an expression for $R_n(\la^{(i)})$, when adding a node on the row $i$ of the diagram of $\la$. Since we have
\[(1-n)R_n(\la)=h_n((1-n)A_\la),\]
this amounts to compare the alphabet $A_{\la^{(i)}}$ with $A_\la$. From now on we denote $\beta=1-\al$ and $x_i=\la_i-(i-1)/\al$. 

\begin{prop}~\label{pr1}
Given a partition $\la$ and an integer $1\le i \le l(\la)+1$, one has
\[A_{\la^{(i)}}=A_\la+B(x_i),\]
with $B(v)$ the alphabet $\{v+1,v-1/\al\}-\{v,v+1-1/\al\}$.
In other words
\[H_{1/z}(A_{\la^{(i)}})=H_{1/z}(A_\la)\,
\frac{-z+x_i+1-1/\al}{-z+x_i-1/\al} \,\frac{-z+x_i}{-z+x_i+1}.\]
\end{prop}
\begin{proof}
By Theorem~\ref{th3} we have
\[H_{1/z}(A_\la)=\frac{(-z-1/\al+1)_{\la}}{(-z-1/\al)_{\la}}\, \frac{(-z)_{\la}}{(-z+1)_{\la}}.\]
But $(u)_{\la^{(i)}}=(u)_\la (u+x_i)$.
\end{proof}

\begin{prop}~\label{pr2}
For any $n\ge 1$ and any real number $u$ one has
\[h_n(uB(v))=\sum_{\begin{subarray}{c}r\ge 1, s\ge 0\\ 2r+s\le n
\end{subarray}} v^{n-2r-s}\binom{n-1}{2r+s-1}
\binom{r+s-1}{s}\binom{-u}{r}(-1/\al)^{r+s}\beta^s.\]
\end{prop}
\begin{proof}
By the definition (4.2), denoting $y=(z-v)^{-1}$, we have
\begin{equation*}
H_{1/z}(uB(v))=\big(H_{1/z}(B(v))\big)^u=
\left(\frac{1-y(1-1/\al)}{1+y/\al} \,\frac{1}{1-y}\right)^u.
\end{equation*}
Using $\beta=1-\al$ this becomes
\begin{equation*}
H_{1/z}(uB(v))=\left(1-\frac{y^2/\al}{1+y\beta/\al}\right)^{-u}.
\end{equation*}
Applying the binomial development three times, we have
\begin{equation*}
\begin{split}
H_{1/z}(uB(v))&=\sum_{r\ge 0} \binom{-u}{r} 
(-1/\al)^ry^{2r}\left(1+y\beta/\al\right)^{-r}\\
&=\sum_{r,s\ge 0} \binom{-u}{r} \binom{r+s-1}{s}
(-1/\al)^{r+s}\beta^s y^{2r+s}\\
&=\sum_{r,s,t\ge 0} \binom{-u}{r} \binom{r+s-1}{s}
\binom{2r+s+t-1}{t} 
(-1/\al)^{r+s}\beta^s v^tz^{-2r-s-t}.
\end{split}
\end{equation*}
We identify the coefficient of $z^{-n}$.
\end{proof}

\noindent \textit{Important remark:} This result plays a fundamental role in the sequel since it is the ``natural'' way making the parameter $\beta$ enter our scheme. Although the fact that $\beta=1-\al$ is crucial for its proof, this restriction will no longer be used. From now on, $\beta$ will play an independent role.

\begin{prop}~\label{pr3}
For any $n\ge 2$ one has
\[M_n(\la^{(i)})-M_n(\la)=
\sum_{\begin{subarray}{c}r\ge 1, s,t \ge 0 \\ 2r+s+t\le n
\end{subarray}}
x_i^{n-2r-s-t} \binom{n-t-1}{2r+s-1}
\binom{r+s-1}{s}(1/\al)^{r+s}(-\beta)^s M_t(\la).
\]
\end{prop}
\begin{proof}
By (4.1) we have
\[M_n(\la^{(i)})=h_n(A_{\la^{(i)}})=
\sum_{t= 0}^n h_t(A_{\la})\,h_{n-t}(B(x_i)).\]
We apply Proposition~\ref{pr2} with $u=1, v=x_i$.
\end{proof}

\begin{theo}~\label{th6}
For any $n\ge 2$ one has
\begin{multline*}
R_n(\la^{(i)})-R_n(\la)=
\sum_{\begin{subarray}{c}r\ge 1, s,t\ge 0\\ 2r+s+t\le n
\end{subarray}}
x_i^{n-2r-s-t} \binom{n-t-1}{2r+s-1}
\binom{r+s-1}{s}\binom{n-1}{r}(-1/\al)^{r+s}\beta^s \\\times
\sum_{|\sigma|=t} \frac{u_\sigma}{1-n+t} \binom{1-n+t}{l(\sigma)}  R_\sigma(\la).
\end{multline*}
\end{theo}
\begin{proof}
Applying (5.6) and (4.1) we have
\begin{equation*}
\begin{split}
(1-n)R_n(\la^{(i)})&=h_n((1-n)A_{\la^{(i)}})\\
&=\sum_{t= 0}^n h_t((1-n)A_{\la})\,h_{n-t}((1-n)B(x_i)).
\end{split}
\end{equation*}
We then evaluate $h_{n-t}((1-n)B(x_i))$ by Proposition~\ref{pr2} with $u=1-n, v=x_i$, and $h_t((1-n)A_{\la})$ by relation (5.5) with $x=1-n$.
\end{proof}

Expanding the right-hand side in descending powers of $x_i$, the first terms are
\begin{multline*}
R_n(\la^{(i)})-R_n(\la)=(n-1)/\al\, x_i^{n-2} -\binom{n-1}{2} \, \beta/\al^2 \,x_i^{n-3}\\
+(n-1)(n-3)\left((n-2)\beta^2/6\al^3-(n-2)^2/12\al^2
-R_2(\la)/\al\right)\,x_i^{n-4}+\ldots
\end{multline*}
By a product over the parts of $\rho$, Theorem~\ref{th6} yields
\begin{equation}
R_\rho(\la^{(i)})-R_\rho(\la)=
\sum_{k=2l(\rho)}^{|\rho|}
x_i^{|\rho|-k} 
\sum_{|\sigma|\le k-2} b_{k,\sigma}(\rho) R_\sigma(\la),
\end{equation}
where the coefficients $b_{k,\sigma}(\rho)$ are polynomials in $(1/\al, \beta)$ with integer coefficients, explicitly known. We shall need the value of $b_{k,\sigma}(\rho)$ for $|\sigma|=k-2$. 

\begin{prop}~\label{pr4}
In the right-hand side of (8.1) the contribution $|\sigma|= k-2$ is given by
\[\frac{1}{\al} \sum_{p \,\mathrm{part}\, \mathrm{of}\, \rho} m_p(\rho)(p-1) \, R_{\rho \setminus p} (\la)\, \sum_{k=2}^{p}
x_i^{p-k} \sum_{|\nu|=k-2} \binom{k-p-1}{l(\nu)}
u_\nu R_{\nu} (\la).\]
\end{prop}
\begin{proof}
If we associate a triple $(r_i,s_i,t_i),\, 2r_i+s_i+t_i\le i$ to any part $i$ of $\rho$, the condition $|\sigma|= k-2$ reads $\sum_i(2r_i+s_i)=2$, which implies $r_i=s_i=0,t_i=i$ for all $i$ but one, say $p$, for which $r_p=1,s_p=0$. Then $\sigma$ is formed by the union of $\rho \setminus p$ and a partition $\nu$ with $|\nu|=t_p$.
\end{proof}

Finally let us mention a consequence of Theorem~\ref{th6} (which is not true for moments).
\begin{prop}~\label{pr4b}
The quantity $-R_n(\la)$ is a polynomial in $-1/\al$ with nonnegative integer coefficients.
\end{prop}
\begin{proof}By induction on $|\la|$, starting from $\la=(0)$. If the property is true for $R_n(\la)$, it is true for $R_n(\la^{(i)})$.
\end{proof}

\section{An algorithm}

We now present our method and its first consequences. We begin with a result following directly from Theorem~\ref{th5} and the fact~\cite[Proposition 2]{La4} that $\vartheta^{\la}_{\mu}$ is a shifted symmetric polynomial of $\la$.
\begin{prop}~\label{pr5}
Let $\mu$ be a partition with no part  $1$. For any $\la$ with $|\la| \ge |\mu|$, the quantity $\vartheta^{\la}_{\mu}$ may be written as a polynomial $K_\mu$ in the free cumulants $\{R_k(\la), k\ge 2\}$, with coefficients in $\mathbf{Q}(\alpha)$.
\end{prop}

We shall compute this polynomial $K_\mu$ by induction on $|\mu|-l(\mu)$. We assume that 
$K_\nu$ is known for $|\nu|-l(\nu)=d$. We fix a partition $\mu$ with $|\mu|-l(\mu)=d+1$. We write the polynomial $\vartheta^{\la}_{\mu}$ as
\[\vartheta^{\la}_{\mu}=\sum_{\rho} a_\rho^\mu R_\rho(\la),\]
the coefficients $a_\rho^\mu \in \mathbf{Q}(\alpha)$ to be determined.

\begin{itemize}
\item[(i)] In the right-hand side of the second equation (3.6), the partitions $\nu=\mu \setminus 2$ and $\nu=\mu_{\downarrow (r)}$ satisfy $|\nu|-l(\nu)=d$. 
We substitute the known values of $K_{\mu \setminus 2}$ and $K_{\mu_{\downarrow (r)}}$ in this equation. 

\item[(ii)] We apply Theorem~\ref{th6} to the left-hand sides of equations (3.6). We have
\begin{equation*}
\begin{split}
\vartheta^{\la^{(i)}}_{\mu}-\vartheta^{\la}_{\mu}&=\sum_{\rho} a_\rho^\mu \big(R_\rho(\la^{(i)})-R_\rho(\la)\big)\\
&=\sum_{\rho} a_\rho^\mu \sum_{k=2l(\rho)}^{|\rho|}x_i^{|\rho|-k}
\sum_{|\sigma|\le k-2} b_{k,\sigma}(\rho) R_\sigma(\la).
\end{split}
\end{equation*}
From Theorem~\ref{th4} we deduce
\begin{equation*}
\begin{split}
\sum_{i=1}^{l(\la)+1} 
c_i(\la) \, \vartheta^{\la^{(i)}}_\mu&=\vartheta^{\la}_{\mu}+
\sum_{\rho} a_\rho^\mu \sum_{k=2l(\rho)}^{|\rho|}M_{|\rho|-k}(\la)
\sum_{|\sigma|\le k-2} b_{k,\sigma}(\rho) R_\sigma(\la),\\
\sum_{i=1}^{l(\la)+1} 
c_i(\la) \,x_i\, \vartheta^{\la^{(i)}}_\mu&=
\sum_{\rho} a_\rho^\mu \sum_{k=2l(\rho)}^{|\rho|}M_{|\rho|-k+1}(\la)
\sum_{|\sigma|\le k-2} b_{k,\sigma}(\rho) R_\sigma(\la).
\end{split}
\end{equation*}

\item[(iii)] Now we may omit $\la$ for a clearer display. Equations (3.6) read
\begin{eqnarray}
\sum_{\rho} a_\rho^\mu \sum_{k=2l(\rho)}^{|\rho|}M_{|\rho|-k} \,
\sum_{|\sigma|\le k-2} b_{k,\sigma}(\rho) R_\sigma=0,&\\
\nonumber \sum_{\rho} a_\rho^\mu \sum_{k=2l(\rho)}^{|\rho|}M_{|\rho|-k+1} \,
\sum_{|\sigma|\le k-2 } b_{k,\sigma}(\rho) R_\sigma=&\\
2 (\al R_2-|\mu|+2)\,m_2(\mu)\,K_{{\mu \setminus 2}}+
\sum _{r\ge 3} &rm_r(\mu)\,K_{\mu_{\downarrow (r)}}.
\end{eqnarray}

\item[(iv)] We express the moments $M_k$ in terms of the free cumulants by (5.3).

We are left with two equations involving polynomials in the basis of free cumulants. By identification we obtain a linear system in the unknown coefficients $a_\rho^\mu$.

It turns out that this linear system has \textit{one and only one} solution. This fact might appear mysterious, but actually it is not. It is only an alternative formulation of Theorem~\ref{th1}.
\end{itemize}

\noindent \textit{Remarks:}
\textit{(i)} As already emphasized, this algorithm does not use the fact that $\al+\beta=1$, except at the step (ii). Otherwise $\al$ and $\beta$ may be independent parameters. This point will be further investigated in Section 15.

\noindent \textit{(ii)} Our method remains valid for $\al=1$. In this case, it gives a new algorithm to compute classical Kerov polynomials. Up to now these polynomials were only computed through Frobenius formula (see for instance~\cite[Section 3.3.2]{R}).

\begin{prop}~\label{pr6}
\begin{itemize}
\item[(i)] $K_\mu$ is a polynomial in $\{R_2(\la),R_3(\la),\ldots,R_{|\mu|-l(\mu)+2}(\la)\}$. 
\item[(ii)]The highest weight of $K_\mu$ is $|\mu|+l(\mu)$ (with the weight of $R_\rho$ defined as $|\rho|$).
\end{itemize}
\end{prop}
\begin{proof}
Both properties are true for $\mu=(2)$. They are proved by induction. 

Firstly assume that $(i)$ is true for partitions $\nu$ with $|\nu|-l(\nu)=d$. Let $R_{p}$ occur in $K_\mu$ with $|\mu|-l(\mu)=d+1$. By~Theorem~\ref{th6} this produces a term with $M_{p-1}$ (hence $R_{p-1}$) at the left-hand side of (9.2). However by the inductive hypothesis, the right-hand side involves only free cumulants up to $R_{d+2}$. By comparison we have $p-1\le d+2$.

Secondly assume that $(ii)$ is true for partitions $\nu$ with $|\nu|+l(\nu)=d$. Let $p$ be the highest weight of $K_\mu$ with $|\mu|+l(\mu)=d+1$. By the inductive hypothesis, the highest weight of the right-hand side of (9.2) is $d$. But the highest weight of the left-hand side is $p-1$, because it is produced by some term $M_{p-k+1} R_\sigma$ with $|\sigma|=k-2$. We conclude by comparison.
\end{proof}

From now on, we shall omit $\la$ and write $(R_2,R_3,\ldots,R_{|\mu|-l(\mu)+2})$ for the indeterminates of $K_\mu$. As mentioned in the introduction, we conjecture its coefficients to be polynomials in $\al$ with integer coefficients.

\section{The case of rows}

Up to Section 14 we  now restrict to the case of a row partition $\mu=(r)$. Equations (3.6) read
\begin{equation*}
\begin{split}
\sum_{i=1}^{l(\la)+1} 
c_i(\la) \, \vartheta^{\la^{(i)}}_r&=\vartheta^\la_r,\\
\sum_{i=1}^{l(\la)+1} 
c_i(\la) \,\big(\la_i-(i-1)/\al\big) \, \vartheta^{\la^{(i)}}_r&=r\,\vartheta^\la_{r-1}.
\end{split}
\end{equation*}
With $K_{r}=\sum_{|\rho|\le r+1} a_\rho^{(r)} R_\rho$ our algorithm becomes
\begin{eqnarray}
\sum_{|\rho|\le r+1} a_\rho^{(r)} \sum_{k=2l(\rho)}^{|\rho|}M_{|\rho|-k}
\sum_{|\sigma|\le k-2} b_{k,\sigma}(\rho) R_\sigma=0,&\\
\sum_{{|\rho|\le r+1}} a_\rho^{(r)} \sum_{k=2l(\rho)}^{|\rho|}M_{|\rho|-k+1}
\sum_{|\sigma|\le k-2 } b_{k,\sigma}(\rho) R_\sigma=
r \sum_{|\tau|\le r} a_\tau^{(r-1)} R_\tau.&
\end{eqnarray}

In this section we shall identify the terms of weight $r+1$ and $r$ in $K_r$. The following auxiliary lemma will play a crucial role.
\begin{prop}~\label{pr7}
Let $\epsilon=0,1$. We write
\[\sum_{{|\rho|=s}} a_\rho \sum_{k=2l(\rho)}^{s}M_{s-k+\epsilon}\sum_{|\sigma|= k-2 } b_{k,\sigma}(\rho) R_\sigma=
\sum_{|\tau|=s-2+\epsilon} x_\tau R_\tau.\]

(i) If $\epsilon=1$ only partitions $\tau_{\uparrow (p)}=\tau \setminus p\cup (p+1)$ contribute to $x_\tau$. More precisely
\[x_\tau=1/\al \sum_{p \,\mathrm{part}\, \mathrm{of}\, \tau} p\, (m_{p+1}(\tau)+1)\,a_{\tau_{\uparrow (p)}}.\]

(ii) If $\epsilon=0$ only $\tau \cup 2$ contributes to $x_\tau$ and
\[x_\tau=1/\al \,(m_2(\tau)+1) \,a_{\tau\cup 2}.\]
\end{prop}
\begin{proof} In the left-hand sides, applying Proposition~\ref{pr4}, we are led to evaluate
\[\frac{1}{\al} \sum_{p \,\mathrm{part}\, \mathrm{of}\, \rho} m_p(\rho)(p-1) \, R_{\rho \setminus p} \, \sum_{k=2}^{p}
M_{p-k+\epsilon} \sum_{|\nu|=k-2} \binom{k-p-1}{l(\nu)}
u_\nu R_{\nu},\]
with $\epsilon=1,0$. Applying (5.7) and (5.8) the second sum is $R_{p-1}$ for $\epsilon=1$ and $\delta_{p2}$ for $\epsilon=0$.
\end{proof}

\begin{theo}~\label{th7}
The term of weight $r+1$ in $K_r$ is $\al^r R_{r+1}$.
\end{theo}
\begin{proof}
The statement is true for $r=2$. We prove it by induction on $r$. By Proposition~\ref{pr6} the highest weight of $K_r$ is $|\rho|=r+1$.
The term of weight $r$ in the left-hand side of (10.2) is produced by some term $M_{|\rho|-k+1} R_\sigma$. Hence $r\le |\rho|-k+1+k-2= |\rho|-1$, which implies $|\rho|=r+1$ and $|\sigma|=k-2$. Taking terms of highest weight in both sides, we obtain
\begin{equation*}
\sum_{|\rho|=r+1} a_\rho^{(r)} \sum_{k=2l(\rho)}^{r+1}M_{r-k+2} \,
\sum_{|\sigma|= k-2 } b_{k,\sigma}(\rho) R_\sigma= r \,\al^{r-1} R_{r}.
\end{equation*}
Applying Proposition~\ref{pr7}(i) with $s=r+1$ for $\tau=(r)$, we obtain $r \,\al^{r-1}=r\, a_{(r+1)}^{(r)}/\al$, hence $a_{(r+1)}^{(r)}=\al^{r}$.

It remains to prove that $a_{\rho}^{(r)}=0$ for $|\rho|=r+1$ and $l(\rho)\ge 2$. This is done by a second induction on the lowest part of $\rho$. Firstly we treat the case of a lowest part $2$. Consider the term of highest weight in (10.1). By the same argument as above, it has weight $r-1$ and writes
\[\sum_{|\rho|=r+1} a_\rho^{(r)} \sum_{k=2l(\rho)}^{r+1}M_{r-k+1} \,
\sum_{|\sigma|= k-2 } b_{k,\sigma}(\rho) R_\sigma= 0.\]
Applying Proposition~\ref{pr7}(ii) with $s=r+1$ for $\tau=\rho\setminus 2$, we obtain $a_\rho^{(r+1)}=0$.

Then our inductive hypothesis assumes $a_{\nu}^{(r)}=0$ for any $\nu$ with lowest part $q-1$. We fix a partition $\rho$ with lowest part $q$, so that $\tau=\rho_{\downarrow (q)}$ has lowest part $q-1$.
Since $\tau_{\uparrow (q-1)}=\rho$, applying Proposition~\ref{pr7}(i) for $\tau$ we have
\[(q-1)(m_{q}(\tau)+1)\,a_{\rho}^{(r)}+\sum_{\begin{subarray}{c}p \,\mathrm{part}\, \mathrm{of}\, \tau\\ p\neq q-1\end{subarray}} p\, (m_{p+1}(\tau)+1)\,a_{\tau_{\uparrow (p)}}^{(r)}=0.\]
In the sum, all $a_{\tau_{\uparrow (p)}}^{(r)}$ vanish because $\tau_{\uparrow (p)}$ has lowest part $q-1$. Hence $a_{\rho}^{(r)}=0$.
\end{proof}

For a clearer display we now write
\[\mathcal{R}_\rho=\prod_{i\ge 2} 
((i-1)R_i)^{m_i(\rho)}/m_i(\rho)!.\]
\begin{theo}~\label{th8}
The term of weight $r$ in $K_r$ is 
\[\al^{r-1}\beta  \, \frac{r}{2}\,\sum_{|\rho|=r} 
(l(\rho)-1)!\, \mathcal{R}_\rho.\]
\end{theo}
\begin{proof} It is done by induction on $r$. We assume that the property is true for $K_{r-1}$. For $|\rho|=r$ we have to prove 
\[a_\rho^{(r)}=\al^{r-1}\beta  \, \frac{r}{2}\,\frac{u_\rho}{l(\rho)} v_\rho.\] 

Firstly, let us compare the terms of weight $r-1$ on both sides of (10.2). On the left-hand side such a term is produced by $M_{|\rho|-k+1} R_\sigma$. Hence $r-1\le |\rho|-1$ and there are two possibilities, either $|\rho|=r+1, |\sigma|=k-3$ or $|\rho|=r, |\sigma|=k-2$.

(i) The contribution of the case $|\rho|=r+1, |\sigma|=k-3$ to the left-hand side of (10.2) is
\[ \al^r \sum_{k=3}^{r}M_{r-k+2} \,
\sum_{|\sigma|= k-3 } b_{k,\sigma}((r+1)) R_\sigma.\]
But by Theorem~\ref{th6} for $|\sigma|= k-3$ we have
\[b_{k,\sigma}((r+1))=
- \frac{\beta}{\al^2}\, \frac{r}{2} \, (r-k+2) \binom{k-3-r}{l(\sigma)} u_\sigma.\]
Applying (5.10), the contribution is therefore
\[-\al^{r-2}\beta\,\frac{r}{2} \,
\sum_{k=3}^{r}(k-1)\, M_{k-1} \,
\sum_{|\sigma|= r-k } \binom{-k}{l(\sigma)} u_\sigma R_\sigma=
-\al^{r-2}\beta\,\frac{r}{2} \,\sum_{|\tau|=r-1} \, w_\tau \frac{u_\tau}{l(\tau)}R_\tau.\]

(ii) The contribution of the case $|\rho|=r, |\sigma|=k-2$ to the left-hand side of (10.2) is
\[\sum_{|\rho|=r} a_\rho^{(r)} \sum_{k=2l(\rho)}^{r}M_{r-k+1} \,
\sum_{|\sigma|= k-2 } b_{k,\sigma}(\rho) R_\sigma.\]
Applying Proposition~\ref{pr7}(i) with $s=r$, this becomes
\[1/\al \sum_{|\tau|=r-1} \sum_{p \,\mathrm{part}\, \mathrm{of}\, \tau} p\, (m_{p+1}(\tau)+1)\,a_{\tau_{\uparrow (p)}}^{(r)} R_\tau.\]
Finally equating the coefficients of $R_\tau$ with $|\tau|=r-1$ on both sides of (10.2), and applying the inductive hypothesis, we obtain
\begin{equation}
\sum_{p \,\mathrm{part}\, \mathrm{of}\, \tau} p\, (m_{p+1}(\tau)+1)\,a_{\tau_{\uparrow (p)}}^{(r)} =
\al^{r-1}\beta \, \frac{r}{2}\, \frac{u_\tau}{l(\tau)} \Big( w_\tau + (r-1)\,v_\tau \Big).
\end{equation}

Secondly, let us compare the terms of weight $r-2$ on both sides of (10.1). On the left-hand side such a term is produced by $M_{|\rho|-k} R_\sigma$. Hence $r-2\le |\rho|-2$ and again two possibilities, either $|\rho|=r+1, |\sigma|=k-3$ or $|\rho|=r, |\sigma|=k-2$.

By the same argument as above, applying (5.9), the contribution of the case $|\rho|=r+1, |\sigma|=k-3$ to the left-hand side of (10.1) is
\[-\frac{r}{2} \, \al^{r-2}\beta
\sum_{k=2}^{r}(k-1)\, M_{k-2} \,
\sum_{|\sigma|= r-k } \binom{-k}{l(\sigma)} u_\sigma R_\sigma=
-\frac{r}{2} \, \al^{r-2}\beta \sum_{|\tau|=r-2} \, v_\tau u_\tau R_\tau.\]
The contribution of the case $|\rho|=r, |\sigma|=k-2$ to the left-hand side of (10.1) is
\[\sum_{|\rho|=r} a_\rho^{(r)} \sum_{k=2l(\rho)}^{r}M_{r-k} \,
\sum_{|\sigma|= k-2 } b_{k,\sigma}(\rho) R_\sigma.\]
Applying Proposition~\ref{pr7}(ii) with $s=r$, this contribution becomes 
\[1/\al \, \sum_{|\tau|=r-2} \, (m_2(\tau)+1) \, a_{\tau\cup 2}^{(r)}\, R_\tau.\]
Finally equating the coefficients of $R_\tau$ with $|\tau|=r-2$ on both sides of (10.1), we obtain
\begin{equation}
(m_2(\tau)+1)\, a_{\tau\cup 2}^{(r)} -\frac{r}{2} \, \al^{r-1}\beta \, v_\tau u_\tau=0.
\end{equation}

We are now in a position to compute $a_{\rho}^{(r)}$ for $|\rho|=r$. This is done by a second induction on the lowest part of $\rho$. Firstly we treat the case of a lowest part $2$, say $\rho=\tau\cup 2$. The value of $a_{\rho}^{(r)}$ is then given by (10.4), which yields
\[a_{\rho}^{(r)} = \al^{r-1}\beta \,\frac{r}{2} \,\frac{u_\tau}{m_2(\tau)+1}\, v_\tau =\al^{r-1}\beta  \, \frac{r}{2}\,\frac{u_\rho}{l(\rho)} v_\rho.\]
Secondly our inductive hypothesis assumes that $a_{\nu}^{(r)}$ is known for any $\nu$ with lowest part $q-1$. We fix a partition $\rho$ with lowest part $q$, so that $\tau=\rho_{\downarrow (q)}$ has lowest part $q-1$. Since $\tau_{\uparrow (q-1)}=\rho$, (10.3) may be written
\[(q-1)(m_{q}(\tau)+1)\,a_{\rho}^{(r)}=
\al^{r-1}\beta \, \frac{r}{2}\, \frac{u_\tau}{l(\tau)} \Big( w_\tau + (r-1)\,
v_\tau \Big)-\sum_{
\begin{subarray}{c}p \,\mathrm{part}\, \mathrm{of}\, \tau\\ p\neq q-1\end{subarray}} p\, (m_{p+1}(\tau)+1)\,a_{\tau_{\uparrow (p)}}^{(r)}.\]
In the right-hand side, all $a_{\tau_{\uparrow (p)}}^{(r)}$ are known because $\tau_{\uparrow (p)}$ has lowest part $q-1$. Thus induction will be completed if we have the identity
\begin{equation*}
\sum_{p\,\mathrm{part}\, \mathrm{of}\, \tau} p\,(m_{p+1}(\tau)+1)\,\frac{u_{\tau_{\uparrow (p)}}}{l(\tau_{\uparrow (p)})} v_{\tau_{\uparrow (p)}}=
\frac{u_\tau}{l(\tau)} \Big( w_\tau + |\tau|v_\tau \Big).
\end{equation*}
But $l(\tau_{\uparrow (p)})=l(\tau)$, $u_{\tau_{\uparrow (p)}}/u_\tau=m_p(\tau)/(m_{p+1}(\tau)+1)$, and $v_{\tau_{\uparrow (p)}}/v_\tau=p/(p-1)$. Thus it amounts to check the trivial equality
\[\sum_{i}i(m_{i+1}+1)\frac{m_i}{m_{i+1}+1}\frac{i}{i-1}=\sum_{i} im_i\Big(\frac{1}{i-1}+1\Big).\]
\end{proof}

\section{More conjectures}

For $r\le 20$ our computations support the following conjecture.
\begin{conj}\label{conj2}
We have
\[K_r = \sum_{\begin{subarray}{c}0\le j \le i \le r-1 \\ 2i-j \le r-1\end{subarray}}
 \al^{r-i} \beta^j \sum_{|\rho|=r-2i+j+1}
a_{ij}(\rho)  \,  R_\rho,\]
where the coefficients $a_{ij}(\rho)$ are nonnegative integers.
\end{conj}

From now on, we denote $K_r^{(i,j)}$ the coefficient of $\al^{r-i} \beta^j$ in $K_r$, with weight $r-2i+j+1$. Theorems~\ref{th7} and~\ref{th8} give
\[K_r^{(0,0)}=R_{r+1},\quad \quad K_r^{(1,1)}=
\frac{r}{2}\,\sum_{|\rho|=r} 
(l(\rho)-1)!\, \mathcal{R}_\rho.\]
We conjecture that all components $K_r^{(i,j)}$ but one may be described in a unified way, \textit{independent of} $r$.
\begin{conj}\label{conj3}
For any $(i,j)\neq (0,0)$ there exists an inhomogeneous symmetric function $f_{ij}$ having maximal degree $4i-2j-2$ such that
\[K_r^{(i,j)} = r\, \sum_{|\rho|=r-2i+j+1} 
(l(\rho)+2i-j-2)!\, f_{ij}(\rho)  \,  \mathcal{R}_\rho,\]
where $f_{ij}(\rho)$ denotes the value of $f_{ij}$ at the integral vector $\rho$. This symmetric function is $\mathrm{independent}$ of $r$.
\end{conj}

This is trivial for $i=j=1$ since we have $f_{11}=1/2$. It is a highly remarkable fact that the conjectured nonnegativity of the coefficients of $K_r^{(i,j)}$ reflects a \textit{much stronger} property, which appears in the expansion of $f_{ij}$ in terms of monomial symmetric functions.
\begin{conj}\label{conj4}
The inhomogeneous symmetric function $f_{ij}$ may be written
\[f_{ij}= \sum_{|\rho| \le 4i-2j-2} c_\rho^{(i,j)} m_\rho,\]
where the coefficients $c_\rho^{(i,j)}$ are positive rational numbers.
\end{conj}

The nonnegativity of the coefficients of $K_r^{(i,j)}$ is an obvious consequence. We emphasize that the coefficients of $f_{ij}$ in terms of any other classical basis of symmetric finctions \textit{may be negative}.

Conjectures~\ref{conj3} and~\ref{conj4} are supported by extensive computer calculations up to $r= 20$, giving many values of the positive numbers $c_\rho^{(i,j)}$. We now review the more elementary cases.

\subsection*{Example $j=0$} 

This is the only situation appearing in the classical Kerov-Biane-F\'eray framework. In this case one has $\al=1,\beta=0$, so that all components $K_r^{(i,j)}$ disappear for $j\neq 0$. 

This classical framework has been extensively studied. Since the components $K_r^{(i,0)}$ correspond to the grading of the classical Kerov polynomial by weights $|\rho|=r-2i+1$, we may benefit from any result there obtained. 

In particular $K_r^{(1,0)}$ corresponds to the component of weight $r-1$ for the classical Kerov polynomial. Its value was conjectured by Biane~\cite{B2} and different proofs were later given by many authors~\cite{Sn,GR,R,F}.
\begin{theo}~\label{th9}
We have
\[K_r^{(1,0)}=\frac{1}{4} \binom{r+1}{3}\,  \sum_{|\rho|=r-1} 
l(\rho)!\, \mathcal{R}_\rho.\]
\end{theo}

A direct proof might be given, following the line of Theorems~\ref{th7} and~\ref{th8}, but it is rather tedious. The component of weight $r-3$ for the classical Kerov polynomial has been also computed in~\cite{GR,R}. In~\cite{La2} it is written as follows.
\begin{theo}~\label{th10}
We have
\[K_r^{(2,0)}=\frac{1}{5760}\binom{r+1}{3} \sum_{|\rho|=r-3} (l(\rho)+2)! \,f(\rho) \, \mathcal{R}_\rho,\]
with
\[f=3 m_{4}+8 m_{31}+10 m_{22}+16 m_{21^2}+24 m_{1^4}
+20 m_{3}+36 m_{21}+48 m_{1^3}+35 m_{2}+40 m_{1^2}+18 m_{1}.
\]
\end{theo}

For $\al=1$ our conjectures have been already presented in~\cite{La2}, where the conjectural value of $f_{i\,0}$ is given for $i=3,4$. Observe that in~\cite{La2}, $K_r^{(i,0)}$ is written with a factor $\binom {r+1}{3}$ instead of $r$. Thus we have 
$f_{i\,0}=\frac{1}{6}(r+1)(r-1)f_i=\frac{1}{6}(m_1+2i)(m_1+2i-2)f_i$
with $f_i$ given in~\cite[Section 2]{La2}.
 
\subsection*{Example $i=j$} 

We know $f_{11}=1/2$. Other values of $f_{i\,i}$ are conjectured to be
\begin{equation*}
\begin{split}
24 f_{22}&=3 m_{2}+4 m_{1^2}+2 m_{1},\\
2.6! f_{33}&=
15 m_{4}+40 m_{31}+60 m_{22}+90 m_{21^2}+144 m_{1^4}\\
&+60 m_{3}+120 m_{21}+180 m_{1^3}+75 m_{2}+100 m_{1^2}+30 m_{1},\\
6.8! f_{44}&=
105 m_{6}+420 m_{51}+945 m_{42}+1484 m_{41^2}+1176 m_{3^2}+2688 m_{321}\\
&+4368 m_{31^3}
+3906 m_{2^3}+6216 m_{2^21^2}+10224 m_{21^4}+17040 m_{1^6}\\
&+1050 m_{5}+3430 m_{41}+6132 m_{32}+9520 m_{31^2}+13356 m_{2^21}+21168 m_{21^3}\\&+34080 m_{1^5}
+4025 m_{4}+10248 m_{31}+14154 m_{22}+21252 m_{21^2}+32592 m_{1^4}\\
&+7364 m_{3}+13776 m_{21}+19656 m_{1^3}
+6412 m_{2}+8120 m_{1^2}+2128 m_{1}.
\end{split}
\end{equation*}

These values support another conjecture giving the linear terms of $K_r^{(i,i)}$, namely 
\begin{equation*}
K_r^{(i,i)}=|s(r,r-i)| R_{r-i+1} + \mathrm{non}\,\mathrm{linear}\,\mathrm{terms},
\end{equation*} 
with $s(r,i)$ the usual Stirling number of the first kind.

\subsection*{Example $i=j+1$} 

We know $24f_{1\,0}=m_1(m_1+2)=m_2+2 m_{1^2}+2 m_1$. Other values of $f_{i,i-1}$ are conjectured to be
\begin{equation*}
\begin{split}
2.6! f_{21}&=
13 m_{4}+40 m_{31}+55 m_{22}+95 m_{21^2}+162 m_{1^4}\\
&+68 m_{3}+150 m_{21}+240 m_{1^3}+103 m_{2}+150 m_{1^2}+48 m_{1},\\
30.8! f_{32}&=
800 m_{6}+3430 m_{51}+7455 m_{42}+12586 m_{41^2}+9520 m_{3^2}+22540 m_{321}\\
&+38220 m_{31^3}
+31640 m_{2^3}+53620 m_{2^21^2}+91260 m_{21^4}+155850 m_{1^6}\\
&+8920 m_{5}+30870 m_{41}+54320 m_{32}+88816 m_{31^2}+123200 m_{2^21}+202440 m_{21^3}\\&+334200 m_{1^5}
+37480 m_{4}+100380 m_{31}+136990 m_{22}+214746 m_{21^2}+338520 m_{1^4}\\
&+74160 m_{3}+144970 m_{21}+213668 m_{1^3}
+69120 m_{2}+91280 m_{1^2}+24320 m_{1}.
\end{split}
\end{equation*}

\subsection*{Example $i=j+2$} 

Theorem~\ref{th10} gives $f_{2\,0}$. For $f_{31}$ the values of $40 . 11! \, c_\rho^{(3,1)}$ are conjectured to be
\small

\hspace{1 cm}\\
\begin{tabular}{|c|c|c|c|c|c|c|c|}
\hline
8 & 71 & 62 & $61^2$ & 53 & 521 & $51^3$ & $4^2$ \\ \hline
13805 & 80520 & 231550 & 412500 & 418880 & 1034880 & 1830840 & 508662 \\ \hline
431 & $42^2$ & $421^2$ & $41^4$ & $3^22$ & $3^21^2$ & $32^21$ & $321^3$ \\ \hline
1598520 & 2216060 & 3921280 & 6901290 & 2818200 & 4996200 & 6925600 & 12205050 \\ \hline
$31^5$ & $2^4$ & $2^31^2$ & $2^21^4$ & $21^6$ & $1^8$ \\ \cline{1-6}
21411500 & 9617300 & 16938350 & 29737400 & 52047600 & 90902700\\ \cline{1-6}
\end{tabular}

\hspace{0.3 cm}\\
\begin{tabular}{|c|c|c|c|c|c|c|c|}
\hline
7 & 61 & 52 & $51^2$ & 43 & 421 & $41^3$ & $3^21$ \\ \hline
292160 & 1448260 & 3563340 & 6201360 & 5462072 & 13046880 & 22528440 & 16548840 \\ \hline
$32^2$ & $321^2$ & $31^4$ & $2^31$ & $2^21^3$ & $21^5$ & $1^7$\\ 
\cline{1-7}
22705760 & 39264280 & 67457940 & 53941800 & 92776200 & 158867500 & 271263300\\
\cline{1-7}
\end{tabular}

\hspace{0.3 cm}\\
\begin{tabular}{|c|c|c|c|c|c|c|c|c|c|c|}
\hline
6 & 51 & 42 & $41^2$ & $3^2$ & 321 & $31^3$ \\ \hline
2536490 & 10429540 & 21470152 & 36269860 & 27074432 & 62069700 & 103992240 \\ \hline
$2^3$ & $2^21^2$ & $21^4$ & $1^6$ \\ \cline{1-4}
84217980 & 141283560 & 235465890 & 390802500 \\ \cline{1-4}
\end{tabular}

\hspace{0.3 cm}\\
\begin{tabular}{|c|c|c|c|c|c|c|}
\hline
5 & 41 & 32 & $31^2$ & $2^21$ & $21^3$ & $1^5$ \\ \hline
11667920 & 38560412 & 64473112 & 104816800 & 140167720 & 225795680 & 361308200 \\ \hline
\end{tabular}

\hspace{0.3 cm}\\
\begin{tabular}{|c|c|c|c|c|}
\hline
4 & 31 & $2^2$ & $21^2$ & $1^4$ \\ \hline
30608765 & 77913572 & 102006542 & 157306160 & 239692200 \\ \hline
\end{tabular}

\hspace{0.3 cm}\\
\begin{tabular}{|c|c|c|c|c|c|}
\hline
3 & 21 & $1^3$ & 2 & $1^2$ & 1  \\ \hline
45634160 & 84412152 & 119902200 & 35681580 & 45160632 & 11246400\\ \hline
\end{tabular}
\normalsize

\section{$Q$-positivity}

For a better understanding of the previous results, it is necessary to introduce new polynomials $C_i$ and $Q_i$ in the free cumulants. 

Define $C_0=Q_0=1$, $C_1=Q_1=0$ and for any $n \ge 2$,
\[C_n= \sum_{|\rho|=n} l(\rho)! \, \mathcal{R}_\rho,
\quad\quad Q_n= \sum_{|\rho|=n} (l(\rho)-1)! \, \mathcal{R}_\rho.\] 
Using the notation
\[\mathcal{C}_\rho=\prod_{i\ge 2} C_i^{m_i(\rho)}/m_i(\rho)!,
\quad\quad \mathcal{Q}_\rho=\prod_{i\ge 2} Q_i^{m_i(\rho)}/m_i(\rho)!,\]
the correspondence between these three families of polynomials is given by
\begin{equation*}
\begin{split}
C_n&=  \sum_{|\rho|=n} \mathcal{Q}_\rho,\\
Q_n&= \sum_{|\rho|=n} {(-1)}^{l(\rho)} (l(\rho)-1)!  \,  \mathcal{C}_\rho,\\
(1-n)R_{n}&=\sum_{|\rho|=n} {(-1)}^{l(\rho)} \,  \mathcal{Q}_\rho=
\sum_{|\rho|=n} {(-1)}^{l(\rho)} l(\rho)!\,  \mathcal{C}_\rho.
\end{split}
\end{equation*}

These relations are better understood by using symmetric functions.
Actually let $\mathbf{A}_\la$ be the alphabet defined by 
\begin{equation*}
(1-n)R_n=h_n(\mathbf{A}_\la),\quad Q_n=-p_n(\mathbf{A}_\la)/n,\quad C_n={(-1)}^n e_n(\mathbf{A}_\la).
\end{equation*}
We emphasize that in spite of
\[(1-n)R_n(\la)=h_n((1-n)A_\la)=h_n(\mathbf{A}_\la),\]
the connection between the two alphabets $\mathbf{A}_\la$ and $A_\la=I_\la-O_\la$ is still unclear. Therefore $\mathbf{A}_\la$ may only be considered as formal.
Then the previous relations are merely the classical properties~\cite[pp. 25 and 33]{Ma}
\begin{equation*}
\begin{split}
p_n&=-n \sum_{|\rho|=n} {(-1)}^{l(\rho)} \,  u_\rho h_\rho/l(\rho)
=-n \sum_{|\rho|=n} (-1)^{n-l(\rho)}  u_\rho e_\rho/ l(\rho),\\
e_n&= \sum_{|\rho|=n} (-1)^{n-l(\rho)} u_\rho h_\rho= \sum_{|\rho|=n} (-1)^{n-l(\rho)} z_{\rho}^{-1} p_{\rho},\\
h_n&=\sum_{|\rho|=n} z_{\rho}^{-1} p_{\rho}=
\sum_{|\rho|=n} (-1)^{n-l(\rho)} u_\rho e_\rho.
\end{split}
\end{equation*}

Obviously $Q$-positivity implies $R$-positivity. Therefore the following conjecture is \textit{a priori} stronger than the positivity statement of Conjecture~\ref{conj2}. It does not however implies integrality of the $R$-coefficients.
\begin{conj}\label{conj5}
For any $(i,j)\neq (0,0)$ the coefficients  of $K_r^{(i,j)}$ in terms of the indeterminates $Q_k$ are nonnegative rational numbers.
\end{conj}

It is clear that $K_r^{(0,0)}=R_{r+1}$ is not $Q$-positive. But by Theorems~\ref{th8} and~\ref{th9} we have
\[K_r^{(1,1)}=\frac{r}{2} \,Q_r,\quad \quad
K_r^{(1,0)}=\frac{1}{4} \binom{r+1}{3} \sum_{|\rho|=r-1} \mathcal{Q}_\rho.\]
Remarkably this $Q$-positivity appears completely analogous (and possibly equivalent) to the $R$-positivity studied above. We conjecture that all components $K_r^{(i,j)}$ but two may be described in a unified way, \textit{independent of} $r$.
\begin{conj}\label{conj6}
For any $(i,j)$ but $(0,0)$ and $(1,1)$, there exists an inhomogeneous symmetric function $g_{ij}$ having maximal degree $4i-2j-2$ such that
\[K_r^{(i,j)} = r\, \sum_{|\rho|=r-2i+j+1} 
(2i-j-1)^{l(\rho)} \, g_{ij}(\rho)  \,  \mathcal{Q}_\rho,\]
where $g_{ij}(\rho)$ denotes the value of $g_{ij}$ at the integral vector $\rho$. This symmetric function is $\mathrm{independent}$ of $r$.
\end{conj}

As in Section 11 the conjectured positivity of the $Q$-coefficients of $K_r^{(i,j)}$ reflects a stronger property, given by the expansion of $g_{ij}$ in terms of monomial symmetric functions.
\begin{conj}\label{conj7}
The inhomogeneous symmetric function $g_{ij}$ may be written
\[g_{ij}= \sum_{|\rho| \le 4i-2j-2} b_\rho^{(i,j)} m_\rho,\]
where the coefficients $b_\rho^{(i,j)}$ are positive rational numbers.
\end{conj}

The assertion of Conjecture~\ref{conj5} is a direct consequence. Conjectures~\ref{conj6} and~\ref{conj7} are supported by computer calculations up to $r= 20$.

\subsection*{Example $j=0$} 

This is the only situation appearing in the classical framework $\al=1$. It has been investigated in~\cite{La2}, where the values of  $b_\rho^{(i,0)}$ are proved for $i=1,2$ and conjectured for $i=3,4$. There $K_r^{(i,0)}$ is written with a factor $\binom {r+1}{3}$ instead of $r$. Thus we have 
$g_{i\,0}=\frac{1}{6}(r+1)(r-1)g_i=\frac{1}{6}(m_1+2i)(m_1+2i-2)g_i$
with $g_i$ as in~\cite[Section 5]{La2}.

\subsection*{Example $i=j$} 

Values of $g_{i\,i}$ are conjectured to be
\begin{equation*}
\begin{split}
24 g_{22}&=3 m_{2}+2 m_{1^2}+2 m_{1},\\
96 g_{33}&=
m_{4}+2 m_{31}+3 m_{22}+3 m_{21^2}+3 m_{1^4}\\
&+4 m_{3}+6 m_{21}+6 m_{1^3}+5 m_{2}+5 m_{1^2}+2 m_{1},\\
27.4!.6! g_{44}&=
405 m_{6}+1350 m_{51}+2835 m_{42}+3630 m_{41^2}+3348 m_{3^2}+6012 m_{321}\\
&+7512 m_{31^3}
+8370 m_{2^3}+9876 m_{2^21^2}+11592 m_{21^4}+12880 m_{1^6}\\
&+4050 m_{5}+10890 m_{41}+18036 m_{32}+22536 m_{31^2}+29628 m_{2^21}+34776 m_{21^3}\\&+38640 m_{1^5}
+15525 m_{4}+32004 m_{31}+41742 m_{22}+48924 m_{21^2}+53880 m_{1^4}\\
&+28404 m_{3}+42444 m_{21}+45720 m_{1^3}
+24732 m_{2}+25272 m_{1^2}+8208 m_{1}.
\end{split}
\end{equation*}

\subsection*{Example $i=j+1$} 

We know $g_{1\,0}=m_1(m_1+2)/24=(m_2+2 m_{1^2}+2 m_1)/24$.
Other values of $g_{i,i-1}$ are conjectured to be
\begin{equation*}
\begin{split}
4.6! g_{21}&=
26 m_{4}+68 m_{31}+87 m_{22}+123 m_{21^2}+147 m_{1^4}\\
&+136 m_{3}+246 m_{21}+294 m_{1^3}+206 m_{2}+244 m_{1^2}+96 m_{1},\\
486.7! g_{32}&=
3240 m_{6}+12042 m_{51}+24057 m_{42}+35496 m_{41^2}+29808 m_{3^2}+57636 m_{321}\\
&+79848 m_{31^3}
+75618 m_{2^3}+103056 m_{2^21^2}+132336 m_{21^4}+157840 m_{1^6}\\
&+36126 m_{5}+106488 m_{41}+172908 m_{32}+239544 m_{31^2}+309168 m_{2^21}+397008 m_{21^3}\\&+473520 m_{1^5}
+151794 m_{4}+339660 m_{31}+436158 m_{22}+557838 m_{21^2}+659376 m_{1^4}\\
&+300348 m_{3}+482490 m_{21}+557568 m_{1^3}
+279936 m_{2}+306288 m_{1^2}+98496 m_{1}.
\end{split}
\end{equation*}

\subsection*{Example $i=j+2$} 

For $g_{2\,0}$ see~\cite[Section 6]{La2}.
For $g_{31}$ the values of $256 . 10! \, b_\rho^{(3,1)}$ are conjectured to be
\small

\hspace{1 cm}\\
\begin{tabular}{|c|c|c|c|c|c|c|c|}
\hline
8 & 71 & 62 & $61^2$ & 53 & 521 & $51^3$ & $4^2$ \\ \hline
48192 & 254976 & 673056 & 1133712 & 1153152 & 2548656 & 4133592 & 1376256 \\ \hline
431 & $42^2$ & $421^2$ & $41^4$ & $3^22$ & $3^21^2$ & $32^21$ & $321^3$ \\ \hline
3727584 & 4862736 & 7769952 & 12008844 & 5941920 & 9494112 & 12127104 & 18474444 \\ \hline
$31^5$ & $2^4$ & $2^31^2$ & $2^21^4$ & $21^6$ & $1^8$ \\ \cline{1-6}
26908920 & 15434568 & 23152176 & 33262758 & 45158040 & 57215025\\ \cline{1-6}
\end{tabular}

\hspace{0.3 cm}\\
\begin{tabular}{|c|c|c|c|c|c|c|c|}
\hline
7 & 61 & 52 & $51^2$ & 43 & 421 & $41^3$ & $3^21$ \\ \hline
1019904 & 4534848 & 10194624 & 16534368 & 14910336 & 31079808 & 48035376 & 37976448 \\ \hline
$32^2$ & $321^2$ & $31^4$ & $2^31$ & $2^21^3$ & $21^5$ & $1^7$\\ 
\cline{1-7}
48508416 & 73897776 & 107635680 & 92608704 & 133051032 & 180632160 & 228860100\\
\cline{1-7}
\end{tabular}

\hspace{0.3 cm}\\
\begin{tabular}{|c|c|c|c|c|c|c|c|c|c|c|}
\hline
6 & 51 & 42 & $41^2$ & $3^2$ & 321 & $31^3$ \\ \hline
8854656 & 32228736 & 60482976 & 93384432 & 73869696 & 143407392 & 208462320 \\ \hline
$2^3$ & $2^21^2$ & $21^4$ & $1^6$ \\ \cline{1-4}
179456688 & 257259384 & 348078420 & 438810750\\ \cline{1-4}
\end{tabular}

\hspace{0.3 cm}\\
\begin{tabular}{|c|c|c|c|c|c|c|}
\hline
5 & 41 & 32 & $31^2$ & $2^21$ & $21^3$ & $1^5$ \\ \hline
40731648 & 117349056 & 179508096 & 259792320 & 319432032 & 428942160 & 534655800 \\ \hline
\end{tabular}

\hspace{0.3 cm}\\
\begin{tabular}{|c|c|c|c|c|}
\hline
4 & 31 & $2^2$ & $21^2$ & $1^4$ \\ \hline
106852416 & 233134080 & 284840832 & 377240160 & 460393920 \\ \hline
\end{tabular}

\hspace{0.3 cm}\\
\begin{tabular}{|c|c|c|c|c|c|}
\hline
3 & 21 & $1^3$ & 2 & $1^2$ & 1  \\ \hline
159304704 & 248865408 & 292037760 & 124561152 & 134784000 & 39260160\\ \hline
\end{tabular}
\normalsize

\section{$\la$-rings}

In the classical Kerov-Biane framework, Goulden and Rattan~\cite{GR,R} considered the expansion of $K_r$ in terms of the indeterminates $C_k$. They conjectured that the components of weight $r-2i+1$ (which correspond to our $K_r^{i,0}$ when $\al=1$) have nonnegative rational coefficients in terms of the $C_k$'s.

It might be tempting to extend this conjecture to the $C$-expansion of any $K_r^{(i,j)}$ for $(i,j)\neq (0,0)$. However this property is \textit{not true}. When $j\neq 0$ the $C$-coefficients of $K_r^{(i,j)}$ \textit{may be negative}. 

It is already obvious that 
\begin{equation*}
K_r^{(1,1)}=\frac{r}{2}Q_r= \frac{r}{2}\sum_{|\rho|=r} 
{(-1)}^{l(\rho)} (l(\rho)-1)!  \, \mathcal{C}_\rho 
\end{equation*}
is not $C$-positive. But the property fails also for $K_r^{(2,2)}$ which has negative coefficients for $r\ge 5$. This is also the case for $K_r^{(2,1)}$
when $r\ge 8$, and for $K_r^{(4,4)}$ or $K_r^{(3,2)}$ for $r\ge 13$. The $C$-expansion of $K_r$ is therefore very different from the $R$ or $Q$-expansions. 

In the classical framework $\al=1$, we already observed~\cite[Sections 6-8]{La2} that the $R$ and $Q$-expansions of $K_r^{(i,0)}$ are connected through the Cauchy formulas (3.3)-(3.5). This is a general fact.

Actually if we compare both expressions of $K_r^{(i,j)}/r$ the equation
\[\sum_{|\rho|=r-2i+j+1} 
(2i-j-1)^{l(\rho)}\, g_{ij}(\rho)  \,  \mathcal{Q}_\rho=
\sum_{|\rho|=r-2i+j+1} 
(l(\rho)+2i-j-2)!\, f_{ij}(\rho)  \,  \mathcal{R}_\rho\]
is merely the abstract identity
\begin{equation*}
\begin{split}
\sum_{|\rho|=n} (-1)^{l(\rho)}\, g_{ij}(\rho)\,  (2i-j-1)^{l(\rho)} \, z_\rho^{-1}p_\rho(A)&=
\sum_{|\rho|=n} (-1)^{l(\rho)}\, f_{ij}(\rho) \, \frac{(l(\rho)+2i-j-2)!}{\prod_{i} m_i(\rho)!} \,h_\rho(A) \\
&=\sum_{|\rho|=n} f_{ij}(\rho) \, \binom{-2i+j+1}{l(\rho)} u_\rho h_\rho(A),
\end{split}
\end{equation*}
specialized at $A=\mathbf{A}_\la$. Moreover it has obvious links with the Cauchy formulas (3.3)-(3.5)
\[(-1)^ne_n(xA)= 
\sum_{|{\rho }| = n} 
(-1)^{l(\rho)}  \, x^{l(\rho)} \,z_\rho^{-1} p_{\rho}(A)=
\sum_{|\rho|=n} \binom{-x}{l(\rho)} u_\rho h_\rho(A).\]

Therefore it seems plausible that the conjectured positivity of $f_{ij}$ and $g_{ij}$ are \textit{two equivalent properties}, reflecting some abstract pattern of the theory of symmetric functions.

The powerful language of $\la$-rings seems necessary for a better understanding of the interplay between these expansions and Cauchy formulas. Here we shall not give details about $\la$-ring theory, and refer the reader to~\cite[Chapter 2]{Las} (or~\cite[Section 3]{La5} for a short survey). 

The simplest case case $i=j=2$ is already a very interesting example. We have conjectured 
\begin{equation*}
\begin{split}
24 \,f_{22}&=3 m_{2}+4 m_{1^2}+2 m_{1}=p_2+2p_1^2+2p_1,\\
24 \,g_{22}&=3 m_{2}+2 m_{1^2}+2 m_{1}=2p_2+p_1^2+2p_1.
\end{split}
\end{equation*}
The equality
\[K_r^{(2,2)}/r=\sum_{|\rho|=r-1} 
g_{22}(\rho)  \,  \mathcal{Q}_\rho=
\sum_{|\rho|=r-1} 
l(\rho)!\, f_{22}(\rho)  \,  \mathcal{R}_\rho\]
is then the specialization at the alphabet $\mathbf{A}_\la$ of
\[\sum_{|\rho|=n} (-1)^{l(\rho)}\, g_{22}(\rho)\, z_\rho^{-1}p_\rho=
\sum_{|\rho|=n}(-1)^{l(\rho)}\,f_{22}(\rho) u_\rho h_\rho,\]
which is itself a direct consequence of the following identities.
\begin{prop}~\label{pr9}
For any alphabet $A$ and any integer $n$, one has
\begin{equation*}
\begin{split}
\sum_{\begin{subarray}{c}(i,j,k) \in \mathsf{N}^3\\i+j+k=n
\end{subarray}} (-1)^i ih_ie_je_k&=-ne_n,\\
\sum_{\begin{subarray}{c}(i,j,k) \in \mathsf{N}^3\\i+j+k=n
\end{subarray}} (-1)^i i^2 h_ie_je_k 
&=  \sum_{|\rho|=n} (-1)^{n-l(\rho)}\, \big(n^2-2p_2(\rho)\big) \, z_\rho^{-1}p_\rho\\
&= -\sum_{|\rho|=n} (-1)^{n-l(\rho)}\,  p_2(\rho) \,u_\rho h_\rho.
\end{split}
\end{equation*}
\end{prop}
\begin{proof}[Sketch of proof]
If $f$ is a symmetric function, we denote by $f[A]$ its $\la$-ring action on the alphabet $A$, which should not be confused with its evaluation $f(A)$.

We consider a one-variable alphabet $z$ such that $h_k(zA)=z^k h_k(A)$. In $\la$-ring terminology $z$ is an ``element of type 1'' and we have $p_n[-z+2]=-z^n+2$. We start from the two ``Cauchy formulas''
\begin{equation*}
\begin{split}
(-1)^n e_n [(-z+2)A]&= \sum_{|\rho| = n} 
(-1)^{l(\rho)}  z_\rho^{-1} p_\rho [-z+2]\, p_\rho [A],\\
&= \sum_{|\rho| = n} m_\rho [z-2] \, h_\rho [A].
\end{split}
\end{equation*}
If we $z$-differentiate once and fix $z=1$, we obtain the first statement. If we differentiate two times and fix $z=1$, we obtain at the left-hand side 
\[(-1)^n \sum_{i+j+k=n} (-1)^i i(i-1) h_ie_je_k.\]
At the right-hand side, we compute directly
\[\partial^2_z\big(p_\rho[-z+2]\big)\Big|_{z=1}=
\partial^2_z\big(\prod_{i\ge 1} (-z^i+2)^{m_i(\rho)}\big)\Big|_{z=1}
=n^2+n-2p_2(\rho).\]
Since ``augmented'' monomial symmetric functions $\prod_i m_i(\rho)!\,m_\rho$ are integral combinations of power sums~\cite[p. 110]{Ma}, we may obtain similarly
\[\prod_{i\ge 1} m_i(\rho)!\, \partial^2_z
\big(m_{\rho}[z-2]\big)\Big|_{z=1}= (-1)^{l(\rho)} l(\rho)!\big(n-p_2(\rho)\big).\]
Combining both results achieves the proof.
\end{proof}

By specialization of the second statement at $A=\mathbf{A}_\la$, we see that the conjectured values of $f_{22}$ and $g_{22}$ are equivalent with
\[K_r^{(2,2)}=\frac{r}{24}\Big(2r(r-1) C_{r-1}+
\sum_{\begin{subarray}{c}(i,j,k) \in \mathsf{N}^3\\i+j+k=r-1
\end{subarray}}i^2(i-1)R_iC_jC_k\Big),\]
which explains why $K_r^{(2,2)}$ cannot be $C$-positive. 

For a similar example with $K_r^{(2,0)}$ see~\cite[Sections 6-8]{La2}. These two examples remain elementary because they only involve symmetric functions of degree $2$.  

In general the study of $K_r^{(i,j)}$ may be done by differentiating $4i-2j-2$ times the Cauchy formulas
\begin{equation*}
\begin{split}
(-1)^n e_n [(-z+2i-j)A]&= \sum_{|\rho| = n} 
(-1)^{l(\rho)}  z_\rho^{-1} p_\rho [-z+2i-j]\, p_\rho [A],\\
&= \sum_{|\rho| = n} m_\rho [z-2i+j] \, h_\rho [A],
\end{split}
\end{equation*}
and fixing $z=1$. On the right-hand side one must compute
\begin{equation*}
\begin{split}
\partial^k_z\big(p_\rho[-z+2i-j]\big)\Big|_{z=1} &=
(2i-j-1)^{l(\rho)}\, F_k(\rho),\\
\prod_{i\ge 1} m_i(\rho)!\, \partial^k_z
\big(m_{\rho}[z-2i+j]\big)\Big|_{z=1}&=
(-1)^{l(\rho)} (l(\rho)+2i-j-2)! \, G_k(\rho),
\end{split}
\end{equation*}
with $k\le 4i-2j-2$ and $F_k,G_k$ some symmetric functions of degree $k$.

For $i=j=3$ this method shows that the conjectured values of $f_{33}$ and $g_{33}$ are equivalent with the beautiful \textit{nonnegative} expansion
\[K_r^{(3,3)}=\frac{r}{48}
\sum_{\begin{subarray}{c}(i,j) \in \mathsf{N}^2\\i+j=r-2
\end{subarray}}i(i+1)^2(i+2)\,C_iC_j.\]

\section{The general case}

When the partition $\mu$ is not a row, the $R$-coefficients of $K_\mu$ are still conjectured to be polynomials in $(\al,\beta)$ with integer coefficients. However these integers may be negative. For instance 
\[K_{22}=\al^4R_3^2+2\al^3\beta R_3R_2-\al^3(4R_4+2R_2^2)+\al^2\beta^2R_2^2
-10\al^2\beta R_3-(2\al^2+6\al\beta^2)R_2.\]

This fact already appears in the classical framework $\al=1$. In this case, F\'eray~\cite{F} observed that another family of polynomials $\tilde{K}_{\mu}$ can be inductively defined by
\[K_\mu=\sum_{k=1}^{l(\mu)} (-1)^{l(\mu)-k} \sum_{(\nu_1,\ldots, \nu_k)} \prod_{i=1}^k  \tilde{K}_{\nu_i},\]
where the second sum is taken over all decompositions of the $l(\mu)$ parts of $\mu$ into $k$ disjoint partitions $(\nu_1,\cdots,\nu_k)$. For instance one has 
\begin{equation*}
\begin{split}
K_r=\tilde{K}_{r},&\quad \quad K_{r,s}=K_rK_s-\tilde{K}_{r,s},\\
K_{r,s,t}=K_rK_sK_t-&K_r\tilde{K}_{s,t}-K_s\tilde{K}_{r,t}-K_t\tilde{K}_{r,s}+\tilde{K}_{r,s,t}.
\end{split}
\end{equation*}
The first values of $\tilde{K}_{\mu}$ are listed in the introduction and those of $\tilde{K}_{rs}$ for $r+s\le 18$ on a web page~\cite{W}.

For $\al=1$ F\'eray~\cite{F} proved the $R$-coefficients of $\tilde{K}_{\mu}$ to be nonnegative. Remarkably this property seems true when $\al\neq1$. We conjecture the structure of $\tilde{K}_{\mu}$ to be very similar to the one of $K_r$.
\begin{conj}
The highest weight of $\tilde{K}_{\mu}$ is $|\mu|-l(\mu)+2$ and we have
\[\tilde{K}_\mu = \sum_{\begin{subarray}{c}0\le j \le i \le |\mu|-l(\mu) \\ 2i-j \le |\mu|-l(\mu)\end{subarray}}
 \al^{|\mu|-l(\mu)+1-i} \beta^j \sum_{|\rho|=|\mu|-l(\mu)+2-2i+j}
a_{ij}(\rho)  \,  R_{\rho},\]
where the coefficients $a_{ij}(\rho)$ are nonnegative integers.
\end{conj}

However $\tilde{K}_\mu^{(i,j)}$, the coefficient of $\al^{|\mu|-l(\mu)+1-i} \beta^j$, is much more complicated than $K_r^{(i,j)}$. In particular the term of highest weight $\tilde{K}_\mu^{(0,0)}$ is no longer a monomial. Also it is not clear which $\tilde{K}_\mu^{(i,j)}$ have nonnegative $Q$-coefficients. At least $\tilde{K}_\mu^{(0,0)}$ and $\tilde{K}_\mu^{(1,1)}$ may have negative ones. 

\section{Extension to arbitrary $\beta$}

In this article $\beta$ was of course never considered as being independent of $\al$. However, as already emphasized, the restriction $\al+\beta=1$ is totally unnecessary. All the results obtained above may be extended when $\al$ and $\beta$ are \textit{independent} parameters. 

We now present this elegant generalization. Let $\zeta$ and $\eta$ be two real numbers with $\zeta<0$ and $\eta>0$. Define
$$\al=-\frac{1}{\zeta \eta},\qquad \beta=\frac{1}{\zeta}+\frac{1}{\eta}.$$
In other words, $\zeta$ and $\eta$ are the roots of $\al x^2+\beta x-1=0$. When $\beta=1-\al$ as above, we have obviously $\eta=1$ and $\zeta=-1/\al$. 

Our advocated generalization is obtained by appropriate substitutions of $\eta$ and $\zeta$ instead of $1$ and $-1/\al$. Incidentally this makes formulas become very symmetrical.

Firstly let us introduce the $(\zeta,\eta)$-transition measure of any Young diagram. For any partition $\la$, the $(\zeta,\eta)$-content of a node $(i,j)\in \la$ is defined as $(i-1)\zeta+(j-1)\eta$. Accordingly the generalized rising factorial is
\[(x)_{\la}=\prod_{(i,j) \in \la} \left(x+(i-1)\zeta+(j-1)\eta\right).\]

The $(\zeta,\eta)$-contents of the inside and outside corners of the Young diagram of $\la$ define a pair of interlacing sequences
\begin{equation*}
x_k(\zeta,\eta)=\la^\prime_k\,\zeta+(k-1)\eta,\quad \quad
y_k(\zeta,\eta)=\la^\prime_k\,\zeta+k\eta,
\end{equation*}
with $1\le k \le \la_1=d-1$ and $x_d(\zeta,\eta)=\la_1\eta$. Here we maintain the convention (5.1) that $x_i$ and $y_{i-1}$ should be omitted whenever $x_i=y_{i-1}$. This pair has center $0$. 

The $(\zeta,\eta)$-transition measure is a measure $\omega_\la$ on the real line, supported on the set $\{x_1(\zeta,\eta),\ldots,x_d(\zeta,\eta)\}$. It is uniquely described by its moment generating series
\[\mathcal{M}_\la(z)= \frac{1}{z-x_d(\zeta,\eta)}\prod_{i=1}^{d-1}\frac{z-y_i(\zeta,\eta)}{z-x_i(\zeta,\eta)}=z^{-1}H_{1/z}(A_\la),\]
with $A_\la=I_\la-O_\la$ the difference of the alphabets \[I_\la=\{x_1(\zeta,\eta),\ldots,x_d(\zeta,\eta)\},\quad \quad 
O_\la=\{y_1(\zeta,\eta),\ldots, y_{d-1}(\zeta,\eta)\}.\]
The moments of $\omega_\la$ are $M_k(\la)=h_k(A_\la)$ and its free cumulants are $R_k(\la)=(-1)^ke_k^*(A_\la)$.

Exactly as in Section 6, the moment series may be written
\begin{equation}
\mathcal{M}_\la(z)=z^{-1} \,
\frac{(-z+\zeta+\eta)_\la}{(-z+\zeta)_\la}\, 
\frac{(-z)_\la}{(-z+\eta)_\la}.
\end{equation}
This expression has several consequences. Firstly the moments $M_k(\la)$ and the free cumulants $R_k(\la)$ are polynomials in $(\zeta,\eta)$. Secondly for any integer $k\ge 0$ we have
\begin{equation}
M_k(\la)= \sum_{i=1}^{l(\la)+1} 
c_i(\la) \,\big((i-1)\zeta+\la_i\,\eta\big)^k,
\end{equation}
where the weights $c_i(\la)$ are the transition probabilities
\[c_i(\la) = \frac {\zeta}{(l(\la)-i+2)\zeta-\la_i\,\eta}
\prod_{\begin{subarray}{c}j=1 \\ j \neq i\end{subarray}}^{l(\la)+1} 
\frac{(j-i+1)\zeta+(\la_j-\la_i)\eta}
{(j-i)\zeta+(\la_j-\la_i)\eta}.\]
Equivalently we have
\[\mathcal{M}_\la(z)=\sum_{i=1}^{l(\la)+1} 
 \frac{c_i(\la)}{z-(i-1)\zeta-\la_i\,\eta}.\]

Now we have the following fundamental remark. Although moments and free cumulants are polynomials in $(\zeta,\eta)$, \textit{all formulas of our algorithm keep unchanged} because these formulas only involve $\zeta+\eta$ and $\zeta\eta$, hence $\al$ and $\beta$.

This is proved very easily. Denote $x_i=(i-1)\zeta+\la_i\,\eta$. As in Proposition~\ref{pr1}, relation (16.1) implies
\[H_{1/z}(A_{\la^{(i)}})=H_{1/z}(A_\la)\,
\frac{-z+x_i+\zeta+\eta}{-z+x_i+\zeta} \,\frac{-z+x_i}{-z+x_i+\eta}.\]
Equivalently $A_{\la^{(i)}}=A_\la+B(x_i)$, with $B(v)$ the alphabet $\{v+\zeta,v+\eta\}-\{v,v+\zeta+\eta\}$.
But as in Proposition~\ref{pr2}, with $y=(z-v)^{-1}$ we have
\begin{equation*}
\begin{split}
H_{1/z}(uB(v))&=
\left(\frac{1-y(\zeta+\eta)}{1-y\zeta} \,\frac{1}{1-y\eta}\right)^u\\
&=\left(1-\frac{y^2/\al}{1+y\beta/\al}\right)^{-u}.
\end{split}
\end{equation*}
In other words, the symmetric functions of $uB(v)$ depend only on $\al$ and $\beta$. This implies that Propositions~\ref{pr1}-\ref{pr4} and Theorem~\ref{th6} keep formally the same provided the quantity $x_i=\la_i-(i-1)/\al$ is replaced by $x_i=(i-1)\zeta+\la_i\,\eta$.

Here is an elementary example. For $\la=(r,s,t)$ relations (5.4) and (16.2) yield easily 
\begin{equation*}
\begin{split}
R_2(\la)&=-\zeta\eta(r+s+t),\\
R_3(\la)&=-\zeta\eta\big((r+3s+5t)\zeta+(r^2+s^2+t^2)\eta\big),\\
R_4(\la)&=-\zeta\eta\big((r+7s+19t)\zeta^2+(3r^2+6s^2+9t^2+3rs+3st+3rt)\zeta\eta
+(r^3+s^3+t^3)\eta^2\big).
\end{split}
\end{equation*}
But for $1\le i\le 3$ these quantities still satisfy
\begin{equation*}
\begin{split}
\al R_2(\la^{(i)})&=\al R_2(\la)+1,\\
\al^2R_3(\la^{(i)})&=\al^2R_3(\la)+2\al x_i-\beta,\\
\al^3R_4(\la^{(i)})&=\al^3R_4(\la)+3\al^2x_i^2-3\al\beta x_i-3\al^2R_2(\la)+\beta^2-\al,
\end{split}
\end{equation*}
with $x_1=r\eta$, $x_2=\zeta+s\eta$, $x_3=2\zeta+t\eta$.

Incidentally these formulas also show, as in Proposition~\ref{pr4b}, that $-R_k(\la)$ is a polynomial in $(\zeta,\eta)$ with \textit{nonnegative} integer coefficients. This property does not hold for moments.

Since our algorithm keeps formally the same, we may define quantities $\vartheta^\la_\mu(\zeta,\eta)$ inductively by the system
\begin{equation*}
\begin{split}
\sum_{i=1}^{l(\la)+1} 
c_i(\la) \, \vartheta^{\la^{(i)}}_\mu&=\vartheta^\la_\mu,\\
\sum_{i=1}^{l(\la)+1} 
c_i(\la)  \,\big((i-1)\zeta+\la_i\,\eta\big)\,
\vartheta^{\la^{(i)}}_\mu&=2 (|\la|-|\mu|+2)\,m_2(\mu)\,\vartheta^\la_{{\mu \setminus 2}}+
\sum _{r\ge 3}r\,m_r(\mu)\,\vartheta^\la_{\mu_{\downarrow (r)}}.
\end{split}
\end{equation*}
Obviously these quantities $\vartheta^\la_\mu$ depend on $(\zeta,\eta)$. However their transition matrix with the free cumulants depends only on $(\alpha,\beta)$. In this context, all results and conjectures of Sections 8 to 14 keep relevant, without any change. 

Finally let us mention another conjecture, which does not involve free cumulants. Consider the finite alphabet of $(\zeta,\eta)$-contents
\[C_\la=\{(i-1)\zeta+(j-1)\eta,\,(i,j) \in \la\}\]
and its power sums
\[p_k(C_\la) = \sum_{(i,j) \in \la} \big((i-1)\zeta+(j-1)\eta\big)^k\qquad (k\ge 1).\]
Again the quantities $\vartheta^\la_\mu$ and the power sums $p_\rho(C_\la)$ depend on $(\zeta,\eta)$, but their transition matrix depends only on $(\alpha,\beta)$.

\begin{conj} 
The quantity $\vartheta^\la_\mu(\zeta,\eta)$ is a polynomial in power sums $p_0(C_\la)=|\la|$ and $\{p_k(C_\la),k\ge 1\}$. Once written in terms of $\binom{|\la|}{r}$ instead of $|\la|^r$, its coefficients are polynomials in $(\alpha,\beta)$ with integer coefficients. 
\end{conj}
For instance we have
\begin{equation*}
\begin{split}
\vartheta_2^\la&=2\al\, p_1(C_\la),\\
\vartheta_3^\la&=3\al^2\, p_2(C_\la)+3\al\beta\,p_1(C_\la)-3\al\binom{|\la|}{2},\\
\vartheta_{2,2}^\la& =-12\al^2\, p_2(C_\la)+4\al^2\,p_{11}(C_\la)+8\al\beta\,p_1(C_\la)+8\al\binom{|\la|}{2},\\
\vartheta_4^\la&=4\al^3\, p_3(C_\la)+12\al^2\beta\,p_2(C_\la)+\big(8\al\beta^2-4\al^2(2|\la|-3)\big)\,p_1(C_\la)-8\al\beta\binom{|\la|}{2}.
\end{split}
\end{equation*}

\section{Final remarks}

We conclude by two remarks. Firstly the existence of quantities $\vartheta^\la_\mu$ associated with two independent parameters $(\al,\beta)$ may lead to a generalization of Jack polynomials. Actually these quantities $\vartheta^\la_\mu$ may be used in a sum extending (3.1). However the precise form of this extension is not yet obvious. 

Secondly the existence of a combinatorial scheme, underlying the theory of Jack polynomials, has been suspected for a long time, though it remains mysterious. Our results give strong evidence for the existence of a pattern involving free cumulants, with $\al$ and $\beta$ playing similar roles. Such a combinatorial interpretation has been recently obtained by F\'eray~\cite{F} for $\al=1$. We expect a generalization to Jack polynomials.

\end{document}